\documentclass[12pt]{amsart}

\usepackage{amsmath,amssymb,amsthm,enumerate,color}

\usepackage[capitalise]{cleveref}
\crefname{theorem}{Theorem}{Theorems}
\crefname{lemma}{Lemma}{Lemmas}
\crefname{corollary}{Corollary}{Corollaries}
\crefname{proposition}{Proposition}{Propositions}
\crefname{conjecture}{Conjecture}{Conjectures}
\crefname{question}{Question}{Questions}
\crefname{definition}{Definition}{Definitions}
\crefname{example}{Example}{Examples}
\crefname{remark}{Remark}{Remarks}
\crefname{question}{Question}{Questions}
\crefname{enumi}{}{}
\crefname{equation}{}{}

\newtheorem{theorem}{Theorem}[section]

\newtheorem{lemma}[theorem]{Lemma}
\newtheorem{proposition}[theorem]{Proposition}

\theoremstyle{definition}

\theoremstyle{remark}
\newtheorem{remark}[theorem]{Remark}

\newcommand{\vf}{\varphi}
\newcommand{\vfd}{\dot{\varphi}}

\global\long\def\oneb{\bar{1}}

\newcommand{\beq}{\begin{equation}}
\newcommand{\eeq}{\end{equation}}

\newcommand{\norm}{|\!|}
\newcommand{\bnorm}{\big|\!\big|}

\DeclareMathOperator{\im}{Im}

\def\sideremark#1{\ifvmode\leavevmode\fi\vadjust{\vbox to0pt{\vss
 \hbox to 0pt{\hskip\hsize\hskip1em
 \vbox{\hsize2.7cm\tiny\raggedright\pretolerance10000
  \noindent #1\hfill}\hss}\vbox to8pt{\vfil}\vss}}}

\begin{document}
\title[]{Obstruction Flat Rigidity of the CR $3$-Sphere}
\author{Sean N. Curry}
\address{Department of Mathematics, Oklahoma State University, Stillwater, OK 74078-5061}
\email{sean.curry@okstate.edu} 
\author{Peter Ebenfelt}
\address{Department of Mathematics, University of California at San Diego, La Jolla, CA 92093-0112}
\email{pebenfel@math.ucsd.edu}

\date{\today}
\thanks{The second author was supported in part by the NSF grant DMS-1900955.}


\begin{abstract} 
On a bounded strictly pseudoconvex domain in $\mathbb{C}^n$, $n >1$, the smoothness of the Cheng-Yau solution to Fefferman’s complex Monge-Amp\`ere equation up to the boundary is obstructed by a local curvature invariant of the boundary, the CR obstruction density $\mathcal{O}$. While local examples of obstruction flat CR manifolds are plentiful, the only known compact examples are the spherical CR manifolds. We consider the obstruction flatness problem for small deformations of the standard CR $3$-sphere. That rigidity holds for the CR sphere was previously known (in all dimensions) for the case of embeddable CR structures, where it also holds at the infinitesimal level. In the $3$-dimensional case, however, a CR structure need not be embeddable. While in the nonembeddable case we may no longer interpret the obstruction density $\mathcal{O}$ in terms of the boundary regularity of Fefferman's equation (or the logarithmic singularity of the Bergman kernel) the equation $\mathcal{O}\equiv 0$ is still of great interest, e.g., since it corresponds to the Bach flat equation of conformal gravity for the Fefferman space of the CR structure (a conformal Lorentzian $4$-manifold).
 Unlike in the embeddable case, it turns out that in the nonembeddable case there is an infinite dimensional space of solutions to the linearized obstruction flatness equation on the standard CR $3$-sphere and this space defines a natural complement to the tangent space of the embeddable deformations. In spite of this, we show that the CR $3$-sphere does not admit nontrivial obstruction flat deformations, embeddable or nonembeddable. 

\end{abstract}


\maketitle

\section{Introduction and Main Results}

In this paper we continue our study of compact obstruction flat CR $3$-manifolds begun in \cite{CE2018-obstruction-flatI,CE2019-obstruction-flatII} by considering the problem via deformation theory of the standard CR $3$-sphere in the space of abstract CR structures. A major tool is the modified version of the Cheng-Lee slice theorem from the authors' recent paper \cite{CE2020-deformations-and-embeddings}. We recall that in \cite{Fefferman1976, Fefferman1979} Fefferman proposed the study of the asymptotic expansion of the Bergman kernel of a strictly pseudoconvex domain in terms of CR invariants of the boundary. In this context obstruction flatness is the same as the vanishing of the boundary trace of the coefficient of the log term in the Fefferman expansion of the Bergman kernel along the diagonal.


To be more precise and explain the terminology, let $\Omega\subset \mathbb{C}^n$, $n>1$, be a bounded strictly pseudoconvex domain with smooth boundary $\partial \Omega$. There are several interrelated approaches to studying the CR geometry of $\partial \Omega$, or, equivalently, the biholomorphic geometry of $\Omega$. In \cite{Fefferman1976, Fefferman1979} Fefferman suggested as an approach to this the study of the boundary asymptotic expansion for the solution of the Dirichlet problem
\begin{equation} \label{eqn:FeffermanComplexMongeAmpere}
\left\{ \begin{array}{l}
\mathcal{J}(u):= (-1)^n \,\mathrm{det} \left( \begin{array}{ c c} u & u_{z^{\bar{k}}}\\ u_{z^{j}} & u_{z^{j}z^{\bar{k}}}\end{array} \right) = 1\;\,\mathrm{in}\;\, \Omega,\\
u=0 \;\,\mathrm{on}\;\, \partial \Omega
\end{array}\right.
\end{equation}
with $u>0$, and the CR boundary invariants which thereby arise; this turned out to be a very fruitful approach, see, e.g., \cite{BaileyEastwoodGraham1994,Hirachi2000,Hirachi2006}. The equation \cref{eqn:FeffermanComplexMongeAmpere} governs the existence of a unique complete K\"ahler-Einstein metric on $\Omega$ with K\"ahler potential $v=-\log(u)$. Fefferman \cite{Fefferman1976} showed that there is always a smooth approximate solution $\rho$ satisfying $\mathcal{J}(\rho)=1+O(\rho^{n+1})$, and that $\rho$ is unique mod $O(\rho^{n+2})$; we call such a $\rho$ a \emph{Fefferman defining function} for $\Omega$. Subsequently, Cheng and Yau \cite{ChengYau1980} proved the existence of a unique solution $u$ to Fefferman's equation \cref{eqn:FeffermanComplexMongeAmpere} which is $C^{\infty}$ in $\Omega$ and $C^{n+1}$ (but not in general $C^{\infty}$) up to the boundary. The precise boundary behavior of $u$ was uncovered by Lee and Melrose \cite{LeeMelrose1982} who showed that $u$ has an asymptotic expansion of the form
\begin{equation}\label{eqn:Lee-Melrose-asymptotics}
u \sim \rho \left(\eta_0 +  \eta_1 \rho^{n+1}\log \rho + \eta_2 (\rho^{n+1}\log \rho)^2 + \cdots), \quad \eta_k \in C^{\infty}(\overline{\Omega}\right)
\end{equation}
where $\rho$ is a Fefferman defining function (so that $\eta_0=1 \mod O(\rho^{n+1})$). The presence of log terms in the expansion explains the failure of smoothness of $u$ up to the boundary. While the solution $u$ is only uniquely determined globally, Graham \cite{Graham1987a, Graham1987b} showed that the coefficients $\eta_k$ mod $O(\rho^{n+1})$ are locally uniquely determined by $\partial \Omega$ (and independent of the choice of Fefferman defining function $\rho$). Moreover, he showed that if the coefficient $\eta_1$ of the first log term vanishes on $\partial \Omega$ then $\eta_k$ vanishes to infinite order at the boundary for all $k \geq 1$. Thus $\eta_1|_{\partial\Omega}$ is precisely the obstruction to $C^{\infty}$ boundary regularity of the Cheng-Yau solution to Fefferman's equation. The local invariant 
\begin{equation}\mathcal{O}:=\eta_1|_{\partial\Omega}
\end{equation}
 of the boundary $\partial \Omega$ is called the \emph{obstruction function} or \emph{obstruction density} (since it actually transforms as a density), and being a CR invariant it can be defined on an abstract strictly pseudoconvex CR manifold of any dimension; for the remainder of this paper will use the term CR manifold to mean strictly pseudoconvex CR manifold. If $M$ is a CR manifold for which the obstruction density vanishes, we say that $M$ is \emph{obstruction flat}.

In the abstract setting the obstruction density still arises naturally as an obstruction to smooth boundary regularity as follows. In \cite{Fefferman1976} Fefferman showed that the chains (distinguished curves in CR geometries introduced by Chern-Moser and E.\ Cartan) can be obtained by projecting the null geodesics from a circle bundle over the CR manifold which carries a natural conformal Lorentzian metric; this circle bundle is known as the \emph{Fefferman space} of the CR structure. The Fefferman space is a CR invariant in the sense that any CR diffeomorphism lifts to a conformal diffeomorphism of the corresponding Fefferman spaces. The Fefferman space construction generalizes in a natural way to abstract CR manifolds \cite{BurnsDiederichShnider1977,Lee1986,Farris1986,NurowskiSparling2003,CapGover2008}. On any conformal manifold (of any signature) there is a local conformal invariant, closely related to the CR obstruction function, known as the Fefferman-Graham obstruction tensor, which obstructs the formal solvability (and hence smooth boundary regularity) for the Poincar\'e-Einstein extension problem and the conformal ambient metric construction of \cite{FeffermanGraham1984,FeffermanGraham2012}; roughly speaking these problems are real analogs of \cref{eqn:FeffermanComplexMongeAmpere}. The notion of ambient metric for a conformal structure generalizes the notion of ambient metric for a CR manifold introduced by Fefferman (in the embedded case) in \cite{Fefferman1976,Fefferman1979}; in particular, the Fefferman-Graham obstruction tensor of a Fefferman space has only one nonvanishing component (in a natural frame), which is precisely the pullback of the CR obstruction density of the underlying CR manifold (see, e.g., \cite{GrahamHirachi2008,NurowskiSparling2003}). 

In $3$ dimensions, the CR obstruction function of a CR manifold therefore corresponds to the Bach tensor of its Fefferman space. In particular, the obstruction flat equation for a CR $3$-manifold corresponds to the Bach flat equation of conformal gravity (see, e.g., \cite{Maldacena2011}) for the corresponding $4$-dimensional Lorentzian Fefferman space (the Bach flat equation is a 4th order conformally invariant equation whose solutions include Einstein metrics). For the standard CR structure on $S^3$, the Fefferman space is $S^3 \times S^1$ with the conformal class of the standard product Lorentzian metric (whose universal cover is the conformal static Einstein universe $S^3\times \mathbb{R}$). CR deformations of the standard $S^3$ therefore correspond to periodic deformations of the Lorentzian conformal structure on $S^3\times \mathbb{R}$ that are also universal covers of Fefferman spaces (equivalently, have special unitary conformal holonomy \cite{CapGover2008}). Since the Anti-de Sitter, de Sitter, and Minkowski spaces all embed conformally in $S^3\times \mathbb{R}$, by considering obstruction flat deformations of the standard $S^3$ we are also considering Bach flat deformations of these spaces subject to the ansatz that our deformation is locally a Fefferman space. That is, our problem can be considered as a dimensional reduction of the Bach flat equation for $4$-dimensional Lorentzian metrics to $3$ dimensions.

If a CR manifold $M$ is locally CR equivalent to the unit sphere, then we say that $M$ is \emph{(locally) spherical}. 
For $\Omega$ the unit ball in $\mathbb{C}^n$ the solution to Fefferman's equation is $u=1-\norm z\norm^2$, which is smooth up to the boundary, hence the obstruction function vanishes for the unit sphere $\mathbb{S}^{2n-1}\subset \mathbb{C}^n$, $n>1$. It follows that any spherical CR structure is obstruction flat. On the other hand, by \cite[Proposition 4.14]{Graham1987a} there are also \emph{local} non-spherical (real analytic) obstruction flat strictly pseudoconvex hypersurfaces in $\mathbb{C}^n$, for any $n>1$ (cf.\ also \cite{NurowskiPlabanski2001} for an explicit example of a non-spherical noncompact Bach flat Fefferman space in $4$ dimensions). It is natural to ask if there are non-spherical solutions of the global obstruction flatness problem for compact CR manifolds. In the case of embeddable structures, there are no such non-spherical solutions in a large $C^1$ open neighborhood of the standard CR $3$-sphere, see \cite{CE2019-obstruction-flatII}. In this case, rigidity holds even at the infinitesimal level \cite{CE2018-obstruction-flatI}. On the other hand, in the nonembeddable case there are nontrivial solutions to the linearized obstruction flatness equation on the standard CR sphere:

\begin{theorem}\label{thm:infinitesimal-solutions}
There is an infinite dimensional space $H^1_{\mathcal{O}}$ of nontrivial infinitesimal deformations solving the linearized obstruction flatness equation on the standard CR $3$-sphere. Moreover, this space is a complement to the tangent space of the embeddable deformations in the space of all infinitesimal deformations.
\end{theorem}
In fact, if $\mathfrak{D} \cong C^{\infty}(S^3,\mathbb{C})$ denotes the space of all $C^{\infty}$ deformation tensors on the standard CR $3$-sphere, then $H^1_{\mathcal{O}} = \mathfrak{D}_0^{\perp}$, the $L^2$ orthogonal complement of the tangent space $\mathfrak{D}_0$ to the embeddable deformations; furthermore, in \cite{CE2020-deformations-and-embeddings} it is shown that there is a linear subspace of $\mathfrak{D}$ which gives a slice for the space of CR structures near the standard CR sphere and such that $H^1_{\mathcal{O}} = \mathfrak{D}_0^{\perp}$ is the $L^2$ orthogonal complement of the space of embeddable CR structures in the slice. Here we are fixing the underlying contact structure on the CR sphere when considering deformations, which is no loss of generality by Gray's classical theorem \cite{Gray1959}. By \cite{CE2020-deformations-and-embeddings} the infinitesimal deformations in $H^1_{\mathcal{O}}$, modulo the linearized action of the finite dimensional CR automorphism group of the standard CR $3$-sphere, are CR inequivalent.

Nevertheless, it turns out that none of these solutions to the linearized problem integrate to solutions of the nonlinear problem.

\begin{theorem} \label{thm:sphere-is-rigid}
The standard CR $3$-sphere is rigid as an obstruction flat CR $3$-manifold. 
\end{theorem}
In other words, there is an open neighborhood $U$ of the standard CR $3$-sphere in the space of CR structures on $S^3$ such that a CR structure in $U$ is obstruction flat if and only if it is CR equivalent to the standard CR $3$-sphere.
Here the topology is the $C^k$ topology for any sufficiently large $k$ (e.g., $k=10$ will do). \cref{thm:sphere-is-rigid} follows immediately from \cref{thm:slice-thm} (cf.\ \cite[Theorem 1.2]{CE2020-deformations-and-embeddings}) and \cref{thm:rigidity} below. This result substantially generalizes the rigidity result for obstruction flat embeddable deformations of the CR 3-sphere that follows from \cite[Corollary 1.4]{CE2019-obstruction-flatII}, since near the standard CR sphere the space of embeddable CR structures is a (Fr\'echet) submanifold with dense complement in the space of abstract CR structures \cite{CE2020-deformations-and-embeddings}. Rigidity in this more general situation is perhaps surprising in light of \cref{thm:infinitesimal-solutions}.  We note that the method of proof is completely different from \cite{CE2019-obstruction-flatII}, which relied on the existence of holomorphic vector fields in the ambient space.

Moreover, we have the following formal rigidity result.

\begin{theorem} \label{thm:formal-nonsolvability}
Let $\vfd \in H^1_{\mathcal{O}} \setminus \{ 0 \}$, a nontrivial solution of the linearized obstruction flatness equation on the standard CR sphere. Then there does not exist a deformation tensor $\ddot\vf$ such that the family $\vf(t) = \vfd \, t + \ddot\vf\, t^2/2$ satisfies the obstruction flatness equation to second order at $t=0$.
\end{theorem}
\begin{remark}
If $\vf(t) = \vfd \, t + \ddot\vf\, t^2/2$ is as in \cref{thm:formal-nonsolvability} then the corresponding family of CR obstruction functions are of the form $\mathcal{O}(t)= \ddot{\mathcal{O}}t^2/2 + O(t^3)$ where $\ddot{\mathcal{O}}\neq 0$. The corresponding family of CR curvatures (Cartan umbilical tensors) $Q(t) = \dot{Q}t + O(t^2)$ has $\dot{Q}\neq 0$. This is in contrast to the case of embeddable deformations of the standard CR sphere, where the vanishing order of $Q(t)$ must match that of $\mathcal{O}(t)$ (see \cite[Theorem 1.2]{CE2018-obstruction-flatI}). Note that the $\vf(t)$ in \cref{thm:formal-nonsolvability} are nonembeddable (for small $t$) since $\vfd \in H^1_{\mathcal{O}}\setminus\{0\}$ and  $H^1_{\mathcal{O}}=\mathfrak{D}_0^{\perp}$.
\end{remark}
We note that the linearized obstruction operator $\mathfrak{D}\to C^{\infty}(S^3,\mathbb{R})$ is not surjective. We denote its cokernel by $H^2_{\mathcal{O}}$; in terms of spherical harmonics (recalled in \cref{sec:spherical-harmonics} below) the space $H^2_{\mathcal{O}}$ may be identified with space of $f\in C^{\infty}(S^3,\mathbb{R})$ whose spherical harmonic expansion $f= \sum_{p,q} f_{p,q}$ satisfies $f_{p,q}=0$ for $p,q\geq 2$. We also denote by $\mathfrak{D}_{BE}$ the (Burns-Epstein) space of deformation tensors $\vf = \sum_{p,q} \vf_{p,q}$ such that $\vf_{p,q}=0$ unless $q\geq p+4$. Note that $H^1_{\mathcal{O}} = \mathfrak{D}_0^{\perp}$ is $L^2$ orthogonal to $\mathfrak{D}_{BE}$.
\begin{theorem}\label{thm:partial-solvability-intro}
There is a neighborhood $U$ of $0$ in $H^1_{\mathcal{O}}$ such that for any $\vf_0 \in U$ there is a $\psi\in \mathfrak{D}_{BE}$ such that the deformation tensor $\vf = \vf_0 + \psi$ satisfies $\mathcal{O}(\vf) \equiv 0 \mod H^2_{\mathcal{O}}$.
\end{theorem}
A more precise version of \cref{thm:partial-solvability-intro} (in which $\psi$ is uniquely determined) is given in \cref{thm:partial-solvability}.
\begin{remark}
See \cref{rem:H1-H2} for an explanation of the notation $H^1_{\mathcal{O}}$ and $H^2_{\mathcal{O}}$. By the argument leading to \cref{thm:partial-solvability} we get the existence of a Kuranishi map $\Psi: H^1_{\mathcal{O}} \to H^2_{\mathcal{O}}$ defined near $0$ whose zero set is locally isomorphic with the set of solutions of $\mathcal{O}\equiv 0$; but \cref{thm:rigidity} (\cref{thm:sphere-is-rigid}) shows that $0$ is an isolated point of $\Psi^{-1}(0)$, see \cref{rem:Kuranishi}. 
\end{remark}

The main tools used in this paper are the author's modified version of the Cheng-Lee slice theorem \cite{CE2020-deformations-and-embeddings} and a careful analysis of the relevant linearized operators on the standard CR $3$-sphere as well estimates for the nonlinear part of the obstruction flatness equation near the CR sphere in terms of the spherical harmonic decomposition of a CR deformation tensor in the slice.

The paper is organized as follows. In \cref{sec:pseudohermitian,sec:CR-sphere} we recall some basic material on pseudohermitian structures and CR invariants, and the standard pseudohermitian structure on the CR $3$-sphere. In \cref{sec:deformations} we consider deformations of the standard CR $3$-sphere with its standard contact form, and derive certain key formulae for the pseudohermitian connection and curvature of the deformed CR structure. \cref{sec:spherical-harmonics} recalls the spherical harmonics and presents the version of the modified slice theorem that we will be using. In \cref{sec:linear-theory} we present the linear theory for the CR curvature and obstruction function operators. We determine the kernel and cokernel of the linearized obstruction operator on the standard CR $3$-sphere; in particular, we prove \cref{prop:kernel-of-DO} which together with the slice theorem implies \cref{thm:infinitesimal-solutions}. In \cref{sec:partial-solvability} we establish the partial solvability of the obstruction flatness equation (i.e.\ solvability modulo $H^2_{\mathcal{O}}$) including \cref{thm:partial-solvability} and its proof. In \cref{sec:formal-rigidity} we discuss formal rigidity, stating and proving \cref{thm:formal-nonsolvability-inequality} which implies \cref{thm:formal-nonsolvability}. We conclude in \cref{sec:general-rigidity-result} with a proof of \cref{thm:rigidity}, which together with the slice theorem proves our main result, \cref{thm:sphere-is-rigid}.


\section{Psedohermitian Structures and the Tanaka-Webster Connection} \label{sec:pseudohermitian}
In this section we recall some basic material on pseudohermitian structures on a CR $3$-manifold; for a more detailed exposition see, e.g., \cite{Webster1978,Lee1988,CE2018-obstruction-flatI}. Let $(\theta, \theta^1, \theta^{\bar 1})$ be an admissible coframe for $(M,H,J)$. Then $h_{1\oneb}$ is defined by 
\beq\label{eqn:levi-h-defn}
d\theta = i h_{1\oneb}\theta^1\wedge\theta^{\oneb}.
\eeq
The Tanaka-Webster connection form $\omega_1{}^1$ and the pseudohermitian torsion $A_{\oneb}{}^1$ are defined by
\beq\label{eqn:d-theta1}
d\theta^1 = \theta^1\wedge \omega_1{}^1 + A_{\oneb}{}^1\,\theta\wedge\theta^{\oneb}
\eeq
and 
\beq\label{eqn:omega11-plus-conjugate}
\omega_{1}{}^1 + \omega_{\oneb}{}^{\oneb} = h^{1\oneb}dh_{1\oneb}
\eeq
where $\omega_{\oneb}{}^{\oneb} = \overline{\omega_{1}{}^{1}}$ and $h^{1\oneb} = (h_{1\oneb})^{-1}$. We will sometimes refer to the individial components of $\omega_1{}^1$, defined by writing
\beq\label{eqn:omega11-components}
\omega_1{}^1 = \omega_1{}^1{}_0\theta + \omega_1{}^1{}_1\theta^1 + \omega_1{}^1{}_{\oneb}\theta^{\oneb}.
\eeq
The Tanaka-Webster scalar curvature $R$ of $(M,H,J)$ is defined by
\beq\label{eqn:d-omega11}
d\omega_1{}^1 = Rh_{1\oneb}\theta^1\wedge\theta^{\oneb} + (\nabla^1A_{11})\theta^1\wedge \theta - (\nabla^{\oneb}A_{\oneb\oneb})\theta^{\oneb}\wedge\theta
\eeq 
where $\nabla$ denotes the Tanaka-Webster covariant derivative and indices have been raised and lowered using $h_{1\oneb}$. The \emph{Cartan umbilical tensor} is then given by 
\beq\label{eqn:Cartan-Q11}
Q_{11} = -\frac{1}{6}\nabla_1\nabla_1 R - \frac{i}{2}RA_{11} + \nabla_0 A_{11} + \frac{2i}{3}\nabla_1\nabla^1 A_{11}
\eeq
The \emph{obstruction function} or \emph{obstruction density} of $(M,H,J)$ is then given by
\beq\label{eqn:obstruction}
\mathcal{O} = \frac{1}{3}(\nabla^1\nabla^1Q_{11} -iA^{11}Q_{11}).
\eeq
Note that $\mathcal{O}$ is really a weighted CR invariant (a section of a CR density line bundle), see, e.g., \cite{CE2018-obstruction-flatI}; throughout this paper we will be working with a fixed contact form $\theta$ and thus may think of $\mathcal{O}$ as a function, as in \cref{eqn:obstruction}.

\section{The Standard Pseudohermitian Structure on the Sphere} \label{sec:CR-sphere}
Let $S^3$ denote the unit sphere in $\mathbb{C}^2$ and $(S^3, H, J)$ the corresponding (standard) CR structure on the $S^3$. Let $(z,w)$ be the standard coordinates on $\mathbb{C}^2$ and endow $S^3$. Define the $(1,0)$-vector field $Z_1$ by
\beq\label{eqn:Z1-on-S3}
Z_1 = \bar w \frac{\partial}{\partial z} - \bar z \frac{\partial}{\partial w}
\eeq 
and let $Z_{\bar 1} = \overline{Z_1}$. Note that $Z_1$ and $Z_{\bar 1}$ are tangent to $S^3$ and therefore span the holomorphic and antiholomorphic tangent spaces of $S^3$ respectively at each point. We also endow $S^3$ with its standard contact form $\theta$, given by $\theta = i\partial \rho|_{TS^3}$ where $\rho = 1-|z|^2-|w|^2$. Then $\theta = -i(\bar z dz +\bar w dw)$ and $d\theta = i(dz\wedge d\bar z + dw \wedge d\bar w)$, where in both formulae the restriction to $TS^3$ is left implicit. Let $T$ denote the Reeb vector field of $\theta$, i.e. the unique vector field on $S^3$ satisfying $\theta(T)=1$ and $d\theta(T, \,\cdot\,)=0$. It is easy to see that 
\beq\label{eqn:T-on-S3}
T = i\left(z\frac{\partial}{\partial z} + w \frac{\partial}{\partial w}\right) - i\left( \bar z \frac{\partial}{\partial \bar z} + \bar w \frac{\partial}{\partial \bar w}\right).
\eeq
The vector fields $(T, Z_1, Z_{\bar 1})$ define a frame for the complexified tangent space of $S^3$, with coframe $(\theta, \theta^1, \theta^{\bar 1})$. By evaluating $-id\theta = dz\wedge d\bar z + dw \wedge d\bar w$ on $(Z_1,Z_{\oneb})$ we see that $h_{1\oneb} = 1$, i.e.
\beq\label{eqn:dtheta-on-S3}
d\theta = i\theta^1\wedge \theta^{\oneb}. 
\eeq
From the structure equation \cref{eqn:d-theta1} we also obtain that
\beq\label{eqn:sphere-omega-and-A}
\omega_1{}^1 = -2i\theta \quad \text{and} \quad A_{11}=0. 
\eeq
From \cref{eqn:d-omega11} we then have
\beq 
R=2.
\eeq
It follows that $Q_{11} = 0$ and $\mathcal{O}=0$ for the standard CR sphere.

\section{Deformations of the Standard CR Sphere}\label{sec:deformations}

Let $(S^3, H, J)$ denote the standard CR sphere, and $\theta$ and $Z_1$ be as above. We now consider the deformed CR structure on $S^3$ whose underlying holomorphic tangent space is spanned by 
\beq \label{eqn:tilde-Z1}
\tilde{Z}_1 = Z_1+ \vf_1{}^{\oneb} Z_{\oneb},
\eeq
where the deformation tensor $\vf = \vf_1{}^{\oneb}$ is a smooth complex function on $S^3$; note that here we are keeping the underlying contact distribution $H$ fixed, which is no loss of generality for deformed structures homotopic to the standard one by Gray's classical theorem \cite{Gray1959}. We keep the contact form $\theta$ fixed, and thus obtain a coframe $(\theta, \tilde{\theta}^1, \tilde{\theta}^{\oneb})$ dual to $(T, \tilde{Z}_1, \tilde{Z}_{\oneb})$ with
\beq \label{eqn:theta1-tilde}
\tilde{\theta}^1 = \frac{1}{1-|\vf|^2}\left( \theta^1- \vf_{\oneb}{}^{1} \theta^{\oneb} \right),
\eeq
where $\vf_{\oneb}{}^1 = \overline{\vf_1{}^{\oneb}}$ and $|\vf|^2 = \vf_1{}^{\oneb} \vf_{\oneb}{}^1$. We then have $d\theta = i\tilde{h}_{1\oneb}\tilde\theta^1\wedge\tilde\theta^{\oneb}$ where
\beq 
\tilde{h}_{1\oneb}=1-|\vf|^2.
\eeq
Note that we may recover $Z_1$ and $\theta^1$ from $\tilde{Z}_1$ and $\tilde{\theta}^1$ by
\beq
Z_1 =   \frac{1}{1-|\vf|^2} \left(
\tilde{Z}_1 - \vf_1{}^{\oneb} \tilde{Z}_{\oneb} \right)
\eeq
and
\beq
\theta^1 = \tilde\theta^1+ \vf_{\oneb}{}^{1} \tilde\theta^{\oneb}.
\eeq
Let $\tilde{\omega}_1{}^1$ denote the Tanaka-Webster connection form for the modified CR structure in the admissible coframe $(\theta, \tilde{\theta}^1, \tilde{\theta}^{\oneb})$, and let $\tilde{A}_{11}$ denote the corresponding pseudohermitian torsion. We decompose $\tilde{\omega}_1{}^1$ in the frame $(\theta, \tilde{\theta}^1, \tilde{\theta}^{\oneb})$ as 
\beq
\tilde\omega_1{}^1 = \tilde\omega_1{}^1{}_0\tilde\theta + \tilde\omega_1{}^1{}_1\tilde\theta^1 + \tilde\omega_1{}^1{}_{\oneb}\tilde\theta^{\oneb}.
\eeq
In order to solve for these components we compute $d\tilde\theta^1$ using that $d\theta^1 = -iR \theta^1 \wedge \theta$ (where $R=2$),
\begin{align*}
d\tilde\theta^1 & = \frac{d|\vf|^2}{(1-|\vf|^2)^2} \wedge (\theta^1 - \vf_{\oneb}{}^1\theta^{\oneb}) + \frac{1}{1-|\vf|^2} \left( d\theta^1 - d\vf_{\oneb}{}^1 \wedge \theta^{\oneb} - \vf_{\oneb}{}^1 d\theta^{\oneb} \right)\\
& = \frac{d|\vf|^2}{1-|\vf|^2} \wedge \tilde\theta^1  + \frac{1}{1-|\vf|^2} \left( - iR \theta^1\wedge \theta - \vf_{\oneb}{}^1{}_{,1}\theta^1\wedge\theta^{\oneb} - (T\vf_{\oneb}{}^1)\theta\wedge \theta^{\oneb} - iR\vf_{\oneb}{}^1\theta^{\oneb} \wedge\theta \right)\\
& = \frac{d|\vf|^2}{1-|\vf|^2} \wedge \tilde\theta^1  + \frac{1}{1-|\vf|^2} \left( - iR \theta^1\wedge \theta - \vf_{\oneb}{}^1{}_{,1}\theta^1\wedge\theta^{\oneb} - \vf_{\oneb}{}^1{}_{,0}\theta\wedge \theta^{\oneb} + iR\vf_{\oneb}{}^1\theta^{\oneb} \wedge\theta \right)\\
& = \frac{d|\vf|^2}{1-|\vf|^2} \wedge \tilde\theta^1 - iR \tilde\theta^1\wedge \theta + \frac{1}{1-|\vf|^2} \left(  - \vf_{\oneb}{}^1{}_{,1}\theta^1\wedge\theta^{\oneb} - \vf_{\oneb}{}^1{}_{,0}\theta\wedge \theta^{\oneb} \right)\\
& = \frac{d|\vf|^2}{1-|\vf|^2} \wedge \tilde\theta^1 - iR \tilde\theta^1\wedge \theta - \vf_{\oneb}{}^1{}_{,1}\tilde\theta^1\wedge\tilde\theta^{\oneb} - \frac{ \vf_{\oneb}{}^1{}_{,0}}{1-|\vf|^2}\theta\wedge \theta^{\oneb}\\
& = \frac{d|\vf|^2}{1-|\vf|^2} \wedge \tilde\theta^1 - iR \tilde\theta^1\wedge \theta - \vf_{\oneb}{}^1{}_{,1}\tilde\theta^1\wedge\tilde\theta^{\oneb} - \frac{ \vf_{\oneb}{}^1{}_{,0}}{1-|\vf|^2}\theta\wedge \tilde\theta^{\oneb}- \frac{ \vf_1{}^{\oneb}\vf_{\oneb}{}^1{}_{,0}}{1-|\vf|^2}\theta\wedge \tilde\theta^1
\end{align*}
where we have used that $\vf_{\oneb}{}^1{}_{,0} = T\vf_{\oneb}{}^1 + 2\omega_{1}{}^1{}_0\vf_{\oneb}{}^1 = T\vf_{\oneb}{}^1 - 2iR\vf_{\oneb}{}^1$. It follows from the structure equation \cref{eqn:d-theta1} that
\beq \label{eqn:A-tilde}
\tilde A_{\oneb}{}^1 = - \frac{\vf_{\oneb}{}^1{}_{,0}}{1-|\vf|^2}, \quad \text{equivalently } \tilde A_{11} = -\vf_{11,0},
\eeq
and that
\beq
\tilde\omega_1{}^1 = -iR\theta  - \vf_{\oneb}{}^1{}_{,1}\tilde\theta^{\oneb} + \frac{\vf_1{}^{\oneb}\vf_{\oneb}{}^1{}_{,0}}{1-|\vf|^2} \theta - \frac{d|\vf|^2}{1-|\vf|^2}  \quad \mathrm{mod} \; \tilde\theta^1.
\eeq
Combining this last equation with the conjugate equation
\beq 
\tilde\omega_{\oneb}{}^{\oneb} = iR\theta - \vf_1{}^{\oneb}{}_{,\oneb} \tilde\theta^1 + \frac{\vf_{\oneb}{}^1\vf_1{}^{\oneb}{}_{,0}}{1-|\vf|^2} \theta - \frac{d|\vf|^2}{1-|\vf|^2}  \quad \mathrm{mod} \; \tilde\theta^{\oneb},
\eeq
it follows from the condition $\tilde\omega_1{}^1 + \tilde\omega_{\oneb}{}^{\oneb} = \tilde h^{1\oneb} d \tilde h_{1\oneb} = -\frac{d|\vf|^2}{1-|\vf|^2}$ that
\beq
\tilde \omega_1{}^1 = -iR\theta - \vf_{\oneb}{}^1{}_{,1} \tilde\theta^{\oneb} + \vf_1{}^{\oneb}{}_{,\oneb} \tilde\theta^1 -  \frac{\vf_{\oneb}{}^1\vf_1{}^{\oneb}{}_{,0}}{1-|\vf|^2} \theta  - \frac{\tilde Z_{\oneb}|\vf|^2}{1-|\vf|^2}\tilde\theta^{\oneb}.
\eeq
In particular
\beq \label{eqn:omega110}
\tilde{\omega}_1{}^1{}_0 = -2i - \frac{\vf_{\oneb}{}^1\vf_1{}^{\oneb}{}_{,0}}{1-|\vf|^2},
\eeq 
\beq \label{eqn:omega111}
\tilde\omega_1{}^1{}_1=\vf_1{}^{\oneb}{}_{,\oneb},
\eeq
and 
\beq\label{eqn:omega111bar}
 \tilde\omega_1{}^1{}_{\oneb} = - \vf_{\oneb}{}^1{}_{,1} - \frac{\tilde Z_{\oneb}|\vf|^2}{1-|\vf|^2}.
\eeq
Note also that 
\beq
\tilde\omega_{\oneb}{}^{\oneb}{}_{\oneb} = \overline{\tilde\omega_1{}^1{}_1} = \vf_{\oneb}{}^1{}_{,1}.
\eeq
We should point out that our conventions \cref{eqn:tilde-Z1,eqn:theta1-tilde} for the frame and coframe of the deformed structure differ from \cite{CE2020-deformations-and-embeddings} since in \cite{CE2020-deformations-and-embeddings} we normalized the Levi form $\tilde{h}_{1\oneb}$ to be $1$.

In order to compute the scalar curvature $\tilde R$ of the deformed structure we compute $d\tilde \omega_1{}^1$ mod $\theta$:
\begin{multline}
d\tilde \omega_1{}^1 = R h_{1\oneb} \theta^1\wedge \theta^{\oneb} - (\tilde Z_1 \vf_{\oneb}{}^1{}_{,1})\tilde\theta^1\wedge\tilde\theta^{\oneb} - (\tilde Z_{\oneb} \vf_{1}{}^{\oneb}{}_{,\oneb})\tilde\theta^1\wedge\tilde\theta^{\oneb}  \\
 - \frac{\vf_{\oneb}{}^1\vf_1{}^{\oneb}{}_{,0}}{1-|\vf|^2} i \tilde h_{1\oneb} \tilde\theta^1\wedge\tilde\theta^{\oneb} - \tilde Z_1 \frac{\tilde Z_{\oneb}|\vf|^2}{1-|\vf|^2}\tilde\theta^1\wedge\tilde\theta^{\oneb}
 \mod \theta.
\end{multline}
Using that $h_{1\oneb}\theta^1\wedge\theta^{\oneb} = \tilde h_{1\oneb} \tilde\theta^1\wedge\tilde\theta^{\oneb}$ we have
\begin{multline} \label{eqn:R-tilde}
\tilde R= R + \tilde{h}^{1\oneb} \left( 
-i\vf_{\oneb}{}^1\vf_1{}^{\oneb}{}_{,0} 
- \vf_{\oneb}{}^{1}{}_{,11} - \vf_1{}^{\oneb}\vf_{\oneb}{}^{1}{}_{,1\oneb}
- \vf_{1}{}^{\oneb}{}_{,\oneb\oneb} - \vf_{\oneb}{}^{1}\vf_{1}{}^{\oneb}{}_{,\oneb 1} 
\vphantom{- \frac{\tilde{Z}_1|\vf|^2}{1-|\vf|^2}\cdot \frac{\tilde{Z}_{\oneb}|\vf|^2}{1-|\vf|^2} - \frac{\tilde{Z}_1\tilde{Z}_{\oneb}|\vf|^2}{1-|\vf|^2}}
\right. \\
\left. - \frac{\tilde{Z}_1|\vf|^2}{1-|\vf|^2}\cdot \frac{\tilde{Z}_{\oneb}|\vf|^2}{1-|\vf|^2} - \frac{\tilde{Z}_1\tilde{Z}_{\oneb}|\vf|^2}{1-|\vf|^2}
 \right).
\end{multline}
For later use we record the following observation concerning the form of \cref{eqn:R-tilde}:
\begin{lemma} \label{lemma:R-tilde-form}
The scalar curvature $\tilde{R}$ of the deformed pseudohermitian structure is given by $(1-|\vf|^2)^{-3}$ times a polynomial in $\vf_{1}{}^{\oneb}$, $\vf_{\oneb}{}^{1}$ and their $Z_1$ and $Z_{\oneb}$ derivatives up to order $2$. Moreover, for each term in the polynomial the total number of derivatives on the $\vf_{1}{}^{\oneb}$, $\vf_{\oneb}{}^{1}$ factors is at most $2$.
\end{lemma}
\begin{proof}
The result follows easily by writing the covariant derivatives in \cref{eqn:R-tilde} in terms of $Z_1$ and $Z_{\oneb}$ derivatives using \cref{eqn:sphere-omega-and-A} (noting that a Reeb derivative may be expressed in terms of the commutator of $Z_1$ and $Z_{\oneb}$ derivatives) and then factoring $(1-|\vf|^2)^{-3}$ out of the expression, noting that $\tilde{h}^{1\oneb} = (1-|\vf|^2)^{-1}$.
\end{proof}


Let $\tilde Q_{11}$ denote the Cartan umbilical tensor of the deformed CR structure, with respect to the admissible coframe $(\theta,\tilde\theta^1,\tilde\theta^{\oneb})$. From \cref{eqn:R-tilde} and \cref{eqn:Cartan-Q11} one easily computes that the linearization at $\vf_1{}^{\oneb}=0$ of the operator $\mathcal{Q}$ that takes $\vf_1{}^{\oneb}$ to $\tilde{Q}_1{}^{\oneb}$ is given by 
\beq \label{eqn:lin-Q}
D\mathcal{Q} \;:\; 
\vfd_1{}^{\oneb} \mapsto \frac{1}{6} \vfd_1{}^{\oneb}{}_{,}{}^{11}{}_{11} + \frac{1}{6} \vfd_{\oneb}{}^1{}_{,}{}^{\oneb\oneb}{}_{11} - \vfd_1{}^{\oneb}{}_{,00} - \frac{2i}{3}\vfd_1{}^{\oneb}{}_{,0}{}^1{}_1 + \frac{i}{2}R\vfd_1{}^{\oneb}{}_{,0}  
\eeq 
with $R=2$. Aside from this, all that we need to know about $\tilde{Q}_{11}$ is the general form of the nonlinear terms when $\tilde{Q}_{11}$ is expressed in terms of $\vf_{1}{}^{\oneb}$, $\vf_{\oneb}{}^{1}$ and their derivatives up to order $4$. Combining \cref{lemma:R-tilde-form} with \cref{eqn:Cartan-Q11}, \cref{eqn:A-tilde} and \cref{eqn:omega110,eqn:omega111,eqn:omega111bar}, by the Leibniz rule it follows that for the Cartan umbilical tensor we have:
\begin{lemma} \label{lemma:Q11-tilde-form}
With respect to the admissible coframe $(\theta,\tilde\theta^1,\tilde\theta^{\oneb})$, the component $\tilde{Q}_{11}$ of the Cartan umbilical tensor of the deformed CR structure is given by $(1-|\vf|^2)^{-5}$ times a polynomial in $\vf_{1}{}^{\oneb}$, $\vf_{\oneb}{}^{1}$ and their $Z_1$ and $Z_{\oneb}$ derivatives up to order $4$. Moreover, for each term in the polynomial the total number of derivatives on the $\vf_{1}{}^{\oneb}$, $\vf_{\oneb}{}^{1}$ factors is at most $4$.
\end{lemma} 

Let $\tilde\nabla$ denote the Tanaka-Webster connection of the deformed pseudohermitian structure, which has connection form $\tilde\omega_1{}^1$ with respect to the admissible coframe $(\theta,\tilde\theta^1,\tilde\theta^{\oneb})$. We use $\tilde h_{1\oneb}$ to raise and lower indices for objects with a tilde. Then we have
\beq
\tilde{\mathcal{O}} = \frac{1}{3} \left( \tilde\nabla^1\tilde\nabla^1\tilde Q_{11} - i\tilde A^{11} \tilde Q_{11}  \right) = \frac{1}{3} \left( \tilde\nabla_{\oneb}\tilde\nabla_{\oneb}\tilde{Q}^{\oneb\oneb} - i\tilde A_{\oneb\oneb} \tilde Q^{\oneb\oneb}  \right)
\eeq
where
\begin{align} \label{eqn:nabla1-nabla1-Q11}
\tilde\nabla_{\oneb}\tilde\nabla_{\oneb}\tilde{Q}^{\oneb\oneb} 
& = (\tilde Z_{\oneb} + \tilde\omega_{\oneb}{}^{\oneb}{}_{\oneb})
(\tilde Z_{\oneb} + 2\tilde\omega_{\oneb}{}^{\oneb}{}_{\oneb})  \tilde{Q}^{\oneb\oneb} \\
\nonumber
& = (\tilde Z_{\oneb} + \vf_{\oneb}{}^1{}_{,1})
(\tilde Z_{\oneb} + 2\vf_{\oneb}{}^1{}_{,1})  \tilde{Q}^{\oneb\oneb} \\
\nonumber
& = (Z_{\oneb} + \vf_{\oneb}{}^1Z_1 + \vf_{\oneb}{}^1{}_{,1})
(Z_{\oneb} + \vf_{\oneb}{}^1Z_1 + 2\vf_{\oneb}{}^1{}_{,1})  \tilde{Q}^{\oneb\oneb}
\end{align}
and $\tilde{A}_{\oneb\oneb} = -\vf_{\oneb\oneb,0}$. Noting that $\tilde{Q}^{\oneb\oneb} = \tilde{h}^{1\oneb} \tilde{Q}_1{}^{\oneb} = (1-|\vf|^2) \tilde{Q}_1{}^{\oneb}$, it follows immediately that the linearization at $\vf_1{}^{\oneb}=0$ of the operator that takes $\vf_1{}^{\oneb}$ to $\tilde{\mathcal{O}}$ is given by 
\beq
D\mathcal{O}=(Z_{\oneb})^2D\mathcal{Q}.
\eeq 
Besides the linear terms in $\tilde{\mathcal{O}}$, it will suffice for our purposes to consider only the nonlinear terms that appear in $\int_{S^3} \tilde{\mathcal{O}} \,\theta\wedge d\theta$. In particular, note that $\tilde\nabla^1\tilde\nabla^1\tilde Q_{11}=\tilde\nabla_{\oneb}\tilde\nabla_{\oneb}\tilde{Q}^{\oneb\oneb}$ integrates to zero with respect to $\theta\wedge d\theta$ by the divergence formula of \cite{Lee1988}. (Another way to see this is to write \cref{eqn:nabla1-nabla1-Q11} as $\tilde\nabla_{\oneb}\tilde\nabla_{\oneb}\tilde{Q}^{\oneb\oneb}$ equals
\begin{equation*}
Z_{\oneb}\left[
(Z_{\oneb} + \vf_{\oneb}{}^1Z_1 + 2\vf_{\oneb}{}^1{}_{,1})  \tilde{Q}^{\oneb\oneb}\right] + Z_1\left[ \vf_{\oneb}{}^1 
(Z_{\oneb} + \vf_{\oneb}{}^1Z_1 + 2\vf_{\oneb}{}^1{}_{,1})  \tilde{Q}^{\oneb\oneb} \right]
\end{equation*}
and note that the images of $Z_1$ and $Z_{\oneb}$ are both $L^2$ orthogonal to the constants; see \cref{sec:spherical-harmonics} below.) We record this observation in the following lemma:
\begin{lemma}\label{lem:obstruction-integral}
For the deformed CR structure we have
\beq\label{eqn:integral-of-obstruction}
\int_{S^3} \tilde{\mathcal{O}} \, \theta\wedge d\theta = -i\int_{S^3} \tilde{A}_{\oneb}{}^1 \tilde{Q}_1{}^{\oneb} \, \theta\wedge d\theta = i\int_{S^3} \dfrac{\vf_{\oneb}{}^1{}_{,0}\tilde{Q}_1{}^{\oneb} }{1-|\vf|^2} \, \theta\wedge d\theta.
\eeq
\end{lemma}
Our aim is to show that the quantity on the right hand side of \cref{eqn:integral-of-obstruction} obstructs the solvability of the equation  $\tilde{\mathcal{O}} = 0$ when $\vf_1{}^{\oneb}$ is sufficiently small in the Folland-Stein space $H^3_{FS}$. (Recall that $H^s_{FS}$ is the anisotropic Sobolev space of functions having $s$ derivatives in $L^2$, where the derivatives are taken only in $Z_1$ and $Z_{\oneb}$ directions.)
In order to ensure control of the nonlinear terms in the expression for $\tilde{\mathcal{O}}$ we need the following lemma, which is an easy consequence of \cref{lemma:Q11-tilde-form} and the formulae derived in this section.
\begin{lemma} \label{lemma:O-tilde-form}
With respect to the contact form $\theta$, the obstruction density $\tilde{\mathcal{O}}$ of the deformed CR structure is given by $(1-|\vf|^2)^{-9}$ times a polynomial in $\vf_{1}{}^{\oneb}$, $\vf_{\oneb}{}^{1}$ and their $Z_1$ and $Z_{\oneb}$ derivatives up to order $6$. Moreover, for each term in the polynomial the total number of derivatives on the $\vf_{1}{}^{\oneb}$, $\vf_{\oneb}{}^{1}$ factors is at most $6$.
\end{lemma} 

Applying $(Z_1)^2$ to this expression for $\tilde{\mathcal{O}}$ and using the Leibniz rule we have:
\begin{lemma} \label{lemma:Z1Z1O-tilde-form}
With respect to the contact form $\theta$, $(Z_1)^2\tilde{\mathcal{O}}$ is given by $(1-|\vf|^2)^{-11}$ times a polynomial in $\vf_{1}{}^{\oneb}$, $\vf_{\oneb}{}^{1}$ and their $Z_1$ and $Z_{\oneb}$ derivatives up to order $8$. Moreover, for each term in the polynomial the total number of derivatives on the $\vf_{1}{}^{\oneb}$, $\vf_{\oneb}{}^{1}$ factors is at most $8$.
\end{lemma}

\section{Spherical Harmonics and the Slice Theorem} \label{sec:spherical-harmonics}



Our study of the obstruction flatness equation will be greatly illuminated by working in terms of spherical harmonics. We therefore introduce for each $p,q\geq 0$ the spherical harmonic space $H_{p,q}$ of functions on $S^3$ that are the restrictions of harmonic homogeneous polynomials of bidegree $(p,q)$ on $\mathbb{C}^2$. Identifying a deformation tensor $\vf$ with its component function $\vf_1{}^{\oneb}$ with respect to the standard frame on the unit sphere $S^3\subset \mathbb{C}^2$, we denote the $H_{p,q}$ component of $\varphi= \vf_1{}^{\oneb}$ by $\varphi_{p,q} = (\vf_1{}^{\oneb})_{p,q}$, so that the $L^2$ orthogonal spherical harmonic decomposition of $\vf$ is given by $\varphi = \sum_{p,q} \varphi_{p,q}$.

Working in spherical harmonics allows us to introduce the following natural spaces of deformation tensors on $S^3$. Let $\mathfrak{D}_0$ denote the set of all deformation tensors $\varphi$ such that $\varphi_{p,q}=0$ if $q=0,1$; $\mathfrak{D}_0$ is then the space of all infinitesimally embeddable deformations of the unit sphere $S^3\subset\mathbb{C}^2$, in the sense that $\dot{\vf} \in \mathfrak{D}_0$ if and only if there is a smooth family $\vf(t)$ of embeddable deformations of $S^3\subset\mathbb{C}^2$ such that $\vf(0)=0$ and $\left.\frac{d}{dt}\right|_{t=0} \vf(t) = \dot{\vf}$ (see, e.g., \cite{CE2020-deformations-and-embeddings}). Let $\mathfrak{D}_0^{\perp}$ denote the set all deformation tensors $\varphi$ such that $\varphi_{p,q}=0$ unless $q=0,1$; $\mathfrak{D}_0^{\perp}$ then represents directions in which one can deform the sphere for which the deformed structure is ``as far as possible'' from being embeddable. 

Let $\mathfrak{D}_{BE} \subset \mathfrak{D}_0$ denote the set of all deformation tensors $\varphi$ such that $\varphi_{p,q}=0$ if $q < p + 4$ (the ``BE'' here stands for Burns-Epstein \cite{BurnsEpstein1990b} who showed that all sufficiently small deformations $\vf \in \mathfrak{D}_{BE}$ are embeddable in $\mathbb{C}^2$; note that our deformation tensor $\vf$ is the conjugate of Burns and Epstein's). The condition $\vf \in \mathfrak{D}_{BE}$ is natural in that it corresponds to saying that the tensor $\vf_{\oneb}{}^1\theta^{\oneb}\otimes Z_1$ has only nonnegative Fourier coefficients with respect to the standard $S^1$ action on $S^3\subset \mathbb{C}^2$ \cite{Bland1994}. It turns out that any sufficiently small embeddable deformation of the CR sphere can be normalized by a contact diffeomorphism so that its deformation tensor $\vf$ lies in $\mathfrak{D}_{BE}$ (see \cite{Bland1994,CE2020-deformations-and-embeddings}). Moreover, this deformation tensor $\vf$ is unique up to the action of the group $\mathrm{PSU}(2,1)$ on $(S^3,H)$ (i.e. up to the group of CR automorphisms of the standard CR structure on $S^3$) and of the group of $S^1$-equivariant contact diffeomorphisms of $(S^3,H)$. One can further normalize $\vf$ by the action of the $S^1$-equivariant contact diffeomorphisms to lie in the space $\mathfrak{D}'_{BE}$ given by those $\vf \in \mathfrak{D}_{BE}$ that additionally satisfy the reality condition $\im\,((Z_{\oneb})^2 \varphi_{p,p+4}) = 0$ along the critical diagonal (note that deformation tensors whose spherical harmonic decomposition is supported on the critical diagonal $q=p+4$ correspond to the $S^1$-invariant deformations of the standard CR sphere). The representative $\vf\in \mathfrak{D}'_{BE}$ is unique up to the action of $\mathrm{PSU}(2,1)$. For the general case (dropping the assumption that the deformation be embeddable) we have:
\begin{theorem}[\cite{CE2020-deformations-and-embeddings}] \label{thm:slice-thm}
Any sufficiently small deformation $\hat{J}$ of the standard CR $3$-sphere $(S^3,H,J)$ may be normalized by the action of a contact diffeomorphism of $(S^3,H)$ so that it is represented by a deformation tensor $\vf\in\mathfrak{D}'_{BE} \oplus \mathfrak{D}_0^{\perp}$. This deformation tensor is unique up to the action of $PSU(2,1)$.
\end{theorem}

We refer to \cref{thm:slice-thm} as the \emph{slice theorem} since, up to the action of $PSU(2,1)$, it gives a local transverse slice for the space of CR structures on $(S^3,H)$ under the action of the contact diffeomorphism group; \cref{thm:slice-thm} is a modified version of the Cheng-Lee slice theorem \cite{ChengLee1995} that enables one to easily identify the embeddable and nonembeddable structures \cite{CE2020-deformations-and-embeddings}. (The qualification ``up to the action of $PSU(2,1)$'' can be removed if one formulates the result as a local slice theorem for the marked CR structures on $(S^3,H)$ \cite{ChengLee1995,CE2020-deformations-and-embeddings}, but this will not needed here.)  Since the obstruction flatness equation $\tilde{\mathcal{O}}=0$ is diffeomorphism invariant, there is no loss of generality in considering the equation only for the CR structures corresponding to deformation tensors $\vf \in \mathfrak{D}'_{BE} \oplus \mathfrak{D}_0^{\perp}$.

\section{The Linear Theory via Spherical Harmonics}\label{sec:linear-theory}

We start by recording some basic properties of the vector fields $T$, $Z_1$ and $Z_{\oneb}$ when acting on functions.  From \cref{eqn:T-on-S3} we can easily see that $T$ preserves the spherical harmonic spaces $H_{p,q}$ and acts on $H_{p,q}$ by 
\beq
Tu=i(p-q)u
\eeq 
(consistent with the fact that $T$ generates the standard $S^1$ action on $S^3 \subset \mathbb{C}^2$, namely $(z,w)\mapsto (e^{it}z,e^{it}w)$). It is also straightforward to check that $Z_1$ maps $H_{p,q}$ to $H_{p-1,q+1}$ if $p\geq 1$ and acts by zero on $H_{0,q}$. Similarly, $Z_{\oneb}$ maps $H_{p,q}$ to $H_{p+1,q-1}$ if $q\geq 1$ and acts by zero on $H_{p,0}$. Combining these observations it is then easy to check that $Z_1$ maps $H_{p,q}$ to $H_{p-1,q+1}$ isomorphically (with inverse $\frac{1}{p(q+1)}Z_{\oneb}$) when $p\geq 1$, and $Z_{\oneb}$ maps $H_{p,q}$ to $H_{p+1,q-1}$ isomorphically (with inverse $\frac{1}{q(p+1)}Z_1$) when $q\geq 1$.

\begin{remark} \label{rem:Foland-Stein-norms-via-spherical-harmonics}
From the above it follows that the sublaplacian $\Delta_b = Z_1 Z_{\oneb} + Z_{\oneb}Z_1$ acts on each $H_{p,q}$ by $2pq+p+q$. Note that the Folland-Stein Sobolev $s$-norm $\norm u \norm_s$ on $H_{FS}^s$ is equivalent to the norm
\beq
 \norm (1+ \Delta_b)^{s/2} u\norm_{L^2} = \left(\sum_{p,q} (1+p+q+2pq)^{s} \norm u_{p,q}\norm_{L^2}^2 \right)^{1/2}
\eeq
where $u= \sum_{p,q} u_{p,q}$ \cite{FollandStein1974,JerisonLee1987}. We will freely make use of this observation in the following.
\end{remark} 
We will be computing with the pseudohermitian calculus connected with the standard admissible coframe for the standard CR structure on $S^3$, and hence we need to record how $\nabla_0$, $\nabla_1$, and $\nabla_{\oneb}$ act on $\vf=\vf_1{}^{\oneb}$. Since $\omega_1{}^{1} = -2i  \theta$ and $\omega_{\oneb}{}^{\oneb} = 2i  \theta$ we have that $\nabla_1$ is always interchangeable with $Z_1$ and $\nabla_{\oneb}$ is always interchangeable with $Z_{\oneb}$; on the other hand, $\nabla_0$ acts by $T+4i$ on $\vf=\vf_1{}^{\oneb}$ and hence by $i(p-q + 4)$ on the $H_{p,q}$ component of $\vf$, which we record as:
\beq\label{eqn:nabla0-on-varphi}
\nabla_0 \vf_{p,q} = i(p-q + 4)\vf_{p,q}.
\eeq

From these observations combined with \cref{eqn:lin-Q} it is easy to see that the linearization $D\mathcal{Q}$ of the Cartan umbilical tensor at $\vf=0$ maps $\mathfrak{D}_{BE}'$ into $\mathfrak{D}_0$, $\mathfrak{D}_0^{\perp}$ into $\mathfrak{D}_0^{\perp}$, and is injective when restricted to the slice $\mathfrak{D}_{BE}'\oplus \mathfrak{D}_0^{\perp}$. To see the first of these claims, note that if $\vf_1{}^{\oneb} \in \mathfrak{D}_{BE}'$ then $\frac{1}{6} \vf_1{}^{\oneb}{}_{,}{}^{11}{}_{11}  - \vf_1{}^{\oneb}{}_{,00} - \frac{2i}{3}\vf_1{}^{\oneb}{}_{,0}{}^1{}_1 + \frac{i}{2}R\vf_1{}^{\oneb}{}_{,0} \in \mathfrak{D}_{BE} \subset \mathfrak{D}_0$ and $\frac{1}{6} \vf_{\oneb}{}^1{}_{,}{}^{\oneb\oneb}{}_{11} = (Z_1)^2(\frac{1}{6} \vf_{\oneb}{}^1{}_{,}{}^{\oneb\oneb}) \in \mathfrak{D}_0 = \mathrm{im}(Z_1)^2$. That $D\mathcal{Q}$ maps $\mathfrak{D}_0^{\perp}$ into $\mathfrak{D}_0^{\perp}$ follows by noting that if $\vf_1{}^{\oneb} \in \mathfrak{D}_0^{\perp}$ then $\vf_1{}^{\oneb}{}_{,}{}^{11}=0$ (so that $\vf_1{}^{\oneb}{}_{,}{}^{11}{}_{11}=0$ and $\vf_{\oneb}{}^1{}_{,}{}^{\oneb\oneb}{}_{11}=0$) and that the operators $\nabla_0$ and $\nabla_1\nabla^1$ preserve the spherical harmonic spaces $H_{p,q}$. The injectivity of $D\mathcal{Q}$ when restricted to the slice $\mathfrak{D}_{BE}'\oplus \mathfrak{D}_0^{\perp}$ follows by general properties of the CR deformation complex of the standard CR sphere (see, e.g., \cite{CE2018-obstruction-flatI}), but since we will need the computations later we prove this directly below. To establish this we introduce the $L^2$ orthogonal projection $\mathcal{P}_1 : \mathfrak{D}_{BE}'\oplus \mathfrak{D}_0^{\perp} \to \mathfrak{D}_{BE}'$ and show that the maps $\mathcal{P}_1 D\mathcal{Q}: \mathfrak{D}_{BE}' \to \mathfrak{D}_{BE}'$ and $D\mathcal{Q}:\mathfrak{D}_0^{\perp}\to \mathfrak{D}_0^{\perp}$ are both injective. We record how these maps act in the following two lemmas:
\begin{lemma} \label{lem:DQ-on-D0perp}
If $\vf_1{}^{\oneb} \in \mathfrak{D}_0{}^\perp$, then $D\mathcal{Q}$ acts via multiplication by
\begin{multline}
 (p-q+4)^2 + \frac{2}{3}(p+1)q(p-q+4) - (p-q+4)  \\ = \left(p-q+4\right)\left(p-q+3 + \frac{2}{3}q(p+1)\right)
\end{multline}
on the $H_{p,q}$ component of $\vf_1{}^{\oneb}$ ($q=0,1$). That is, $D\mathcal{Q}$ acts on $H_{p,0}$ by $(p+4)(p+3)$ and on $H_{p,1}$ by $\frac{1}{3}(p+3)(5p+8)$. In particular, $D\mathcal{Q}:\mathfrak{D}_0^{\perp}\to \mathfrak{D}_0^{\perp}$ is injective.
\end{lemma}
\begin{proof} The result follows immediately from \cref{eqn:lin-Q}, noting that $\vf_1{}^{\oneb}\in \mathfrak{D}_0{}^\perp = \mathrm{ker}(Z_{\oneb})^2$ implies $\vf_1{}^{\oneb}{}_{,}{}^{11}=0$.
\end{proof}
Using the fact that $(Z_1)^2(Z_{\oneb})^2$ preserves the spherical harmonics and acts by
\beq \label{Z1-squared-Z1b-squared}
(p+1)(p+2)(q-1)q
\eeq
on each $H_{p,q}$, for $\mathcal{P}_1 D\mathcal{Q}: \mathfrak{D}_{BE}' \to \mathfrak{D}_{BE}'$ we have:
\begin{lemma} \label{lem:P1DQ-on-D'BE}
If  $\vf_1{}^{\oneb} \in \mathfrak{D}'_{BE}$, then $\mathcal{P}_1 D\mathcal{Q}$ acts via multiplication  by
\beq \label{eqn:DQ-for-q-geq-p-plus-5}
\frac{1}{6}(p+1)(p+2)(q-1)q + (q-p-4)^2 - \frac{2}{3}(p+1)q(q-p-4) + q-p-4 
\eeq
on the $H_{p,q}$ component of component of $\vf_1{}^{\oneb}$ when $q>p+4$ and by
\beq\label{eqn:DQ-for-q-eq-p-plus-4}
\frac{1}{3}(p+1)(p+2)(q-1)q
\eeq
when $q=p+4$. 
\end{lemma}
\begin{proof} If  $\vf_1{}^{\oneb} \in \mathfrak{D}'_{BE}$, then $\mathcal{P}_1(\vf_{\oneb}{}^1{}_{,}{}^{\oneb\oneb}{}_{11})$ has vanishing spherical harmonic components, except along the critical diagonal ($p=q+4$).
The expression \cref{eqn:DQ-for-q-geq-p-plus-5} therefore follows immediately from \cref{eqn:lin-Q}.  It remains to consider the case when $\vf_1{}^{\oneb} \in \mathfrak{D}'_{BE}\cap H_{p,p+4}$. In this case the reality condition $\im (\vf_{11,}{}^{11}) =0$ (imposed along the critical diagonal) implies that $\frac{1}{6} \vf_1{}^{\oneb}{}_{,}{}^{11}{}_{11} + \frac{1}{6} \vf_{\oneb}{}^1{}_{,}{}^{\oneb\oneb}{}_{11} = \frac{1}{3}\vf_1{}^{\oneb}{}_{,}{}^{11}{}_{11}$. Since the remaining terms in $D\mathcal{Q}(\vf_1{}^{\oneb})$ are zero in this case, we obtain \cref{eqn:DQ-for-q-eq-p-plus-4}.
\end{proof}
It is easy to check that \cref{lem:P1DQ-on-D'BE} implies that $\mathcal{P}_1 D\mathcal{Q}: \mathfrak{D}_{BE}' \to \mathfrak{D}_{BE}'$ is injective. For later use we record the following stronger result, which compares the action of $\mathcal{P}_1 D\mathcal{Q}$ on $\mathfrak{D}_{BE}'$ with the action of the square of the sublaplacian (cf.\ \cref{rem:Foland-Stein-norms-via-spherical-harmonics}):
\begin{lemma}\label{lem:P1DQ-vs-sublaplacian-squared}
There exist positive constants $C_1, C_2$ such that, for $p\geq 0$ and $q>p+4$
\beq \label{eqn:DQ-on-D'BE-bound-1}
C_1 \leq \frac{\frac{1}{6}(p+1)(p+2)(q-1)q + (q-p-4)^2 - \frac{2}{3}(p+1)q(q-p-4) + q-p-4}{(1 + p + q + 2pq)^2} \leq C_2
\eeq
and for $p\geq 0$ and $q=p+4$
\beq \label{eqn:DQ-on-D'BE-bound-2}
C_1 \leq \frac{\frac{1}{3}(p+1)(p+2)(q-1)q}{(1 + p + q + 2pq)^2} \leq C_2.
\eeq
\end{lemma}
\begin{proof} This is a basic exercise in multivariable calculus. The existence of $C_2$ follows from the fact that the denominator in both expressions is larger than $p^2+q^2$, and the numerator is a degree 2 polynomial in $p$ and $q$. The existence of $C_1>0$ small enough such that \cref{eqn:DQ-on-D'BE-bound-2} holds is obvious once we set $q=p+4$ (one obtains a decreasing function for $p\geq 0$ which tends to $1/12$, so any $C_1\leq 1/12$ will do). It is also an easy exercise to see that there exists $C_1>0$ small enough such that \cref{eqn:DQ-on-D'BE-bound-1} holds. One way to do this is to write $S(p,q)$ for the numerator (given by \cref{eqn:DQ-for-q-geq-p-plus-5}) and $T(p,q) = (1 + p + q + 2pq)^2$, and then consider $R(p,q) = S(p,q)-\frac{1}{48}T(p,q)$. Writing $R(p,q) = a(p)q^2 + b(p)q + c(p)$ for quadratic polynomials $a(p)$, $b(p)$, $c(p)$ we have that $a(p)>0$ for all $p$, and it is easy to check that $q\to R(p,q)$ is a positive quadratic for all $p\geq 0$. In particular, for all $p\geq 0$ and $q>p+4$ (indeed, for any $q$) we have $R(p,q)>0$ and hence $S(p,q)/T(p,q) > 1/48$, as required.
\end{proof}
\cref{lem:P1DQ-vs-sublaplacian-squared} shows that, in a precise sense, $\mathcal{P}_1 D\mathcal{Q}:\mathfrak{D}'_{BE} \to \mathfrak{D}'_{BE}$ behaves like $(1+\Delta_b)^2$. In particular, as a map from $\mathfrak{D}'_{BE} \to \mathfrak{D}'_{BE}$ the operator $\mathcal{P}_1 D\mathcal{Q}$ is an injective fourth order operator whose inverse gains $4$ derivatives in Folland-Stein spaces, and preserves each $H_{p,q}$ for $q\geq p+4$. We will also later need the corresponding result for $(Z_1)^2(Z_{\oneb})^2$, which acts on each $H_{p,q}$ by \cref{Z1-squared-Z1b-squared}. It is straightforward to check that: 
\begin{lemma}\label{lem:Z1-squared-Zoneb-squared-in-spherical-harmonics}
There exist positive constants $C_1, C_2$ such that, for $p\geq 0$ and $q\geq p+4$
\beq 
C_1 \leq \frac{(p+1)(p+2)(q-1)q}{(1 + p + q + 2pq)^2} \leq C_2.
\eeq
\end{lemma}
We are now ready to discuss the linearized obstruction operator $D\mathcal{O} = (Z_{\oneb})^2D\mathcal{Q}$. As an easy consequence of the above discussion we have:
\begin{proposition} \label{prop:kernel-of-DO}
The kernel of the linearized obstruction operator $D\mathcal{O}$ restricted to the slice $\mathfrak{D}'_{BE}\oplus \mathfrak{D}_0^{\perp}$ is given by $\mathfrak{D}_0^{\perp}$.
\end{proposition}
\begin{proof}
Since $D\mathcal{O} = (Z_{\oneb})^2D\mathcal{Q}$, this follows immediately from the injectivity of $D\mathcal{Q}$ and the fact that $D\mathcal{Q}$ maps $\mathfrak{D}'_{BE}$ into $\mathfrak{D}_0$ and maps $\mathfrak{D}_0^{\perp}$ into $\mathfrak{D}_0^{\perp} = \mathrm{ker}(Z_{\oneb})^2$.
\end{proof}
Together with \cref{thm:slice-thm}, this establishes \cref{thm:infinitesimal-solutions} from the introduction.

\begin{remark}
As an aside we note that the space $\mathfrak{D}_0^{\perp}$ is also the tangent space (within the slice) at the standard CR sphere to the space of CR structures on $S^3$ that are fillable by self-dual asymptotically complex hyperbolic Einstein metrics on the real $4$-ball \cite{Biquard2005}.
\end{remark}

It is also straightforward to determine the image of the map $D\mathcal{O}: \mathfrak{D}'_{BE}\oplus \mathfrak{D}_0^{\perp} \to C^{\infty}(S^3,\mathbb{R})$ in terms of spherical harmonics. 
\begin{proposition} The image of the map $D\mathcal{O}: \mathfrak{D}'_{BE}\oplus \mathfrak{D}_0^{\perp} \to C^{\infty}(S^3,\mathbb{R})$ is the space of functions $f \in C^{\infty}(S^3,\mathbb{R})$ with spherical harmonic decomposition of the form $\sum_{p,q\geq 2} f_{p,q}$.
\end{proposition}
\begin{proof}
Noting that $D\mathcal{O}$ vanishes on $\mathfrak{D}_0^{\perp}$ and that the image of $D\mathcal{O}$ consists of real functions it is enough to consider the $H_{p,q}$-components of $D\mathcal{O}(\vf)$ for $\vf\in \mathfrak{D}'_{BE}$ with $q\geq p$. The result then easily follows from the above discussion by considering how $(Z_{\oneb})^2 \mathcal{P}_1 D\mathcal{Q}$ acts on $\mathfrak{D}'_{BE}$ ($\mathcal{P}_1 D\mathcal{Q}$ acts injectively and by scalar multiplication on each $H_{p,q}$ component and then $(Z_{\oneb})^2$ maps each $H_{p,q}$ isomorphically to $H_{p+2,q-2}$). Another way to see this result is by noting that $\mathcal{DO}$ is zero on the space $i(Z_1)^2 C^{\infty}(S^3,\mathbb{R})$ of trivial infinitesimal deformation tensors \cite{CE2020-deformations-and-embeddings}, so we can restrict to any slice (complement to the space of trivial deformations) to compute the image of $\mathcal{DO}$; using the Cheng-Lee slice \cite{ChengLee1995} given by those deformation tensors $\vf=\vf_{11}$ such that $\mathrm{Im}\, (\nabla^1\nabla^1\vf_{11}) =0$ the first two terms in the right hand side of \cref{eqn:lin-Q} drop out and the remaining terms preserve each $H_{p,q}$; the action of $D\mathcal{O}$ is obtained by composing \cref{eqn:lin-Q} with $(Z_{\oneb})^2$ and the result follows easily.
\end{proof}

Note that the cokernel $H^2_{\mathcal{O}}$ of $D\mathcal{O}: \mathfrak{D}'_{BE}\oplus \mathfrak{D}_0^{\perp} \to C^{\infty}(S^3,\mathbb{R})$ is nontrivial, and can be identified with the space of functions $f \in C^{\infty}(S^3,\mathbb{R})$  with spherical harmonic decomposition of the form $f=\sum_{p,q} f_{p,q}$ with $f_{p,q}=0$ if $p,q\geq 2$. (The space $H^2_{\mathcal{O}}$ can also be described as the space of all functions of the form $f=\mathrm{Re}(g_0 + g_1 \bar z + g_2 \bar w)$ where $(z,w)$ are the coordinates on $\mathbb{C}^2$ and the $g_j$ are arbitrary smooth CR functions on the unit sphere $S^3\subset \mathbb{C}^2$.) Hence, while on the standard CR $3$-sphere there is an infinite dimensional family $H^1_{\mathcal{O}} = \mathfrak{D}_0^{\perp}$ of solutions to the linearized equation corresponding to $\mathcal{O}\equiv 0$, the problem of obtaining solutions to $\mathcal{O}\equiv 0$ via deformation theory is obstructed by the presence of the (also infinite dimensional) cokernel $H^2_{\mathcal{O}}$ of the linearized operator $D\mathcal{O}: \mathfrak{D}'_{BE}\oplus \mathfrak{D}_0^{\perp} \to C^{\infty}(S^3,\mathbb{R})$.

\begin{remark}\label{rem:H1-H2}
The reason we use the notation $H^1_{\mathcal{O}}$ and $H^2_{\mathcal{O}}$ is that these spaces may thought of as the first and second cohomologies of the deformation complex
\beq
0\rightarrow C^{\infty}(S^3,\mathbb{R}) \rightarrow \mathfrak{D} \rightarrow C^{\infty}(S^3,\mathbb{R})\rightarrow 0,
\eeq
where $\mathfrak{D} \cong C^{\infty}(S^3,\mathbb{C})$ denotes the space of infinitesimal deformation tensors on the standard CR sphere, the second arrow is the linearized action of the contact diffeomorphism group ($f\mapsto iZ_1Z_1 f$) and the third arrow is the linearized obstruction operator $D\mathcal{O}$.
\end{remark}

\section{Partial Solvability}\label{sec:partial-solvability}

In this section we use results of the previous section combined with a standard deformation theory argument to show that the linearized solutions to the obstruction flatness equation on the standard CR sphere do integrate to solutions of the nonlinear equation $\mathcal{O} \equiv 0 \mod H^2_{\mathcal{O}}$, i.e. $\mathcal{O}_{p,q} =0$ for $p,q\geq 2$. More precisely, we shall prove:
\begin{theorem}\label{thm:partial-solvability}
There exist neighborhoods $U$ of $0 \in \mathfrak{D}_0^{\perp}$ and $V$ of $0 \in \mathfrak{D}_{BE}'$ in the $H^6_{FS}$ topology such that for any $\vf_0 \in U$ there is a unique $\psi\in V$ such that the deformation tensor $\vf = \psi + \vf_0$ satisfies $\mathcal{O}(\vf) \equiv 0 \mod H^2_{\mathcal{O}}$.
\end{theorem}

Before we prove this theorem we will collect some useful lemmas. In the following we denote the image of $D\mathcal{O}: \mathfrak{D}'_{BE}\oplus \mathfrak{D}_0^{\perp} \to C^{\infty}(S^3,\mathbb{R})$ by $\mathfrak{Im}$, so that $C^{\infty}(S^3,\mathbb{R}) = \mathfrak{Im} \oplus H^2_{\mathcal{O}}$. We have seen that the linearization $D\mathcal{O}$ of the obstruction function at the standard CR sphere restricts to an invertible linear map from $\mathfrak{D}_{BE}' \to \mathfrak{Im}$ whose inverse gains $6$ derivatives in Folland-Stein spaces. For the proof of \cref{thm:partial-solvability} we need a following slightly more general result given in the following lemma.

Identifying the space of (marked) CR structures on $(S^3,H)$ with the slice $\mathfrak{D}_{BE}'\oplus \mathfrak{D}_0^{\perp}$ we let $D_{\vf_0} \mathcal{O}$ denote the linearization of the CR obstruction function at the CR structure corresponding to $\vf_0 \in \mathfrak{D}_{BE}'\oplus \mathfrak{D}_0^{\perp}$. Let $\mathcal{P}_{\mathfrak{Im}}$ denote the $L^2$ orthogonal projection from $C^{\infty}(S^3,\mathbb{R})= \mathfrak{Im} \oplus H^2_{\mathcal{O}}$ to $\mathfrak{Im}$.
\begin{lemma} \label{lem:linearized-obstruction-at-vf0-estimate}
There is a constant $C>0$ such that for all $\vf_0 \in \mathfrak{D}_0^{\perp}$ sufficiently small one has
\beq\label{eqn:DO-at-vf0-estimate}
C\norm \vfd \norm_{6} \leq \norm \mathcal{P}_{\mathfrak{Im}} D_{\vf_0} \mathcal{O} (\vfd) \norm_{L^2} \leq C^{-1}\norm \vfd \norm_{6}
\eeq
for any $\vfd \in \mathfrak{D}_{BE}'$.
\end{lemma}
\begin{proof} The upper bound in \cref{eqn:DO-at-vf0-estimate} follows from the fact that $D_{\vf_0} \mathcal{O}$ can be expressed as a $6$th order operator involving only $Z_1$ and $Z_{\oneb}$ derivatives (and depends continuously on $\vf_0$). In the case where $\vf_0 =0$ the lower bound follows immediately from \cref{lem:P1DQ-on-D'BE,lem:P1DQ-vs-sublaplacian-squared} and the fact that $D\mathcal{O} = (Z_{\oneb})^2D\mathcal{Q}$ (note that the image of $D\mathcal{Q}|_{\mathfrak{D}_{BE}'}$ is orthogonal to the kernel of $(Z_{\oneb})^2$). From the general form of the obstruction function (\cref{lemma:O-tilde-form}), halving $C$ if necessary, such an estimate will continue to hold so long as $\vf_0$ is sufficiently small in $H^6_{FS}$.
\end{proof}

Setting up for the proof of \cref{thm:partial-solvability} we let $\mathfrak{B}_1$ denote the closure of $\mathfrak{D}_{BE}'$ in the $H_{FS}^6$ norm, and $\mathfrak{B}_2$ the closure of $\mathfrak{Im}$ in the $L^2$ norm; $D\mathcal{O}$ then extends to an isomorphism from $\mathfrak{B}_1$ to $\mathfrak{B}_2$. By a slight abuse of notation we continue to write $\mathcal{P}_{\mathfrak{Im}}$ for the bounded extension of the projection $\mathcal{P}_{\mathfrak{Im}}$ introduced above (i.e. for the $L^2$ orthogonal projection onto $\mathfrak{B}_2$). Given $\vf_0 \in \mathfrak{D}_0^{\perp}$ we let $\mathcal{F}_{\vf_0}:\mathfrak{B}_1 \to \mathfrak{B}_2$ be given by
\beq
\mathcal{F}_{\vf_0}(\psi) = \mathcal{P}_{\mathfrak{Im}} \mathcal{O}(\psi+\vf_0).
\eeq
Then the linearization of $\mathcal{F}_{\vf_0}$ at $\psi =0$ is $\mathcal{L}_{\vf_0} = \mathcal{P}_{\mathfrak{Im}}  D_{\vf_0} \mathcal{O}|_{\mathfrak{B}_1} : \mathfrak{B}_1 \to \mathfrak{B}_2$, which is an isomorphism for $\vf_0$ sufficiently small by \cref{lem:linearized-obstruction-at-vf0-estimate}. Write
\beq
\mathcal{F}_{\vf_0}(\psi) = \mathcal{F}_{\vf_0}(0) +  \mathcal{L}_{\vf_0}(\psi) + \mathcal{N}_{\vf_0}(\psi).
\eeq
Setting $f = \mathcal{L}_{\vf_0}(\psi)$, the partial obstruction flatness equation $\mathcal{F}_{\vf_0}(\psi) =\mathcal{P}_{\mathfrak{Im}} \mathcal{O}(\psi+\vf_0) = 0$ can be written as
\beq
f = - \mathcal{F}_{\vf_0}(0) - \mathcal{N}_{\vf_0}(\mathcal{L}_{\vf_0}^{-1} f).
\eeq

In order to show that the map $f\mapsto - \mathcal{F}_{\vf_0}(0) - \mathcal{N}_{\vf_0}(\mathcal{L}_{\vf_0}^{-1} f)$ has a fixed point we will need the following lemma, which follows from \cref{lemma:O-tilde-form} and the fact that $\mathcal{N}_{\vf_0}(\psi)$ contains no constant or linear terms in $\psi$.
\begin{lemma} \label{lem:nonlinear-estimate-for-Kuranishi-argument}
There exist neighborhoods $U$ of $0 \in \mathfrak{D}_0^{\perp}$ and $V$ of $0 \in \mathfrak{D}_{BE}'$ in the $C^6$ topology and a constant $C>0$ such that 
\beq
\norm \mathcal{N}_{\vf_0}(\psi_1)-\mathcal{N}_{\vf_0}(\psi_2)\norm_{L^2} \leq C \left(\norm \psi_1\norm_6 + \norm \psi_2\norm_6\right)\norm \psi_1 - \psi_2\norm_6
\eeq
for all $\vf_0\in U$ and $\psi_1,\psi_2 \in V$.
\end{lemma}

\begin{proof}[Proof of \cref{thm:partial-solvability}]

Take open neighborhoods $U$ of $0 \in \mathfrak{D}_0^{\perp}$ and $V$ of $0 \in \mathfrak{D}_{BE}'$ as in \cref{lem:nonlinear-estimate-for-Kuranishi-argument}. From \cref{lem:linearized-obstruction-at-vf0-estimate,lem:nonlinear-estimate-for-Kuranishi-argument} it follows that there is $C'>0$ such that
\beq
\norm \mathcal{N}_{\vf_0}(\mathcal{L}_{\vf_0}^{-1} f_1)-\mathcal{N}_{\vf_0}(\mathcal{L}_{\vf_0}^{-1} f_2)\norm_{L^2} \leq C' \left(\norm f_1\norm_{L^2} + \norm f_2\norm_{L^2}\right)\norm f_1 - f_2\norm_{L^2}
\eeq
for all $\vf_0\in U$ and all $f_1,f_2\in \mathfrak{B}_2$ sufficiently small. Hence if $\mathcal{F}_{\vf_0}(0)$ is sufficiently small (which can be ensured by taking $\vf_0$ sufficiently small) there exists $r>0$ such that 
\beq
T_{\vf_0}: f\mapsto - \mathcal{F}_{\vf_0}(0) - \mathcal{N}_{\vf_0}(\mathcal{L}_{\vf_0}^{-1} f)
\eeq
preserves the ball $B_{\mathfrak{B}_2}(0,r)$ of radius $r$ about $0$ in $\mathfrak{B}_2$ and is a contraction mapping on this ball. Shrinking $U$ if necessary, we may assume this is the case for all $\vf_0 \in U$. By the contraction mapping theorem it follows that for each $\vf_0\in U$ the map $T_{\vf_0}$ has a unique fixed point $f_{\vf_0}$ in $B_{\mathfrak{B}_2}(0,r)$, which corresponds to a unique solution $\psi = \mathcal{L}_{\vf_0}^{-1} f_{\vf_0}$ of the equation $\mathcal{F}_{\vf_0}(\psi) =\mathcal{P}_{\mathfrak{Im}} \mathcal{O}(\psi+\vf_0) = 0$ in $\mathcal{L}_{\vf_0}^{-1}\left(B_{\mathfrak{B}_2}(0,r)\right)$. We could just as well have obtained $\psi$ by applying the contraction mapping theorem to the map $\mathcal{L}_{\vf_0}^{-1}\circ T_{\vf_0} \circ\mathcal{L}_{\vf_0}$ and the uniform bound on the operators $\mathcal{L}_{\vf_0}$ given by \cref{lem:linearized-obstruction-at-vf0-estimate} therefore shows that (after shrinking $U$ further if necessary) for each $\vf_0\in U$ the solution $\psi = \mathcal{L}_{\vf_0}^{-1} f_{\vf_0}$ of the equation $\mathcal{F}_{\vf_0}(\psi) = 0$ is the unique solution to this equation in the ball $B_{\mathfrak{B}_1}(0,r') \subseteq \mathfrak{B}_1$ for some fixed $r'>0$ (independent of $\vf_0$). We therefore replace $V$ by $V\cap B_{\mathfrak{B}_1}(0,r')$. This establishes the result with the solutions $\psi = \mathcal{L}_{\vf_0}^{-1} f_{\vf_0}$ having $H^6_{FS}$ regularity. It remains to show that the solutions $\psi$ are in fact $C^{\infty}$ (given that we take $\vf_0$ to be $C^{\infty}$).

To see that $\psi = \mathcal{L}_{\vf_0}^{-1} f_{\vf_0}$ is $C^{\infty}$ for $\vf_0$ sufficiently small we rewrite the equation characterizing $\psi$ as
\beq\label{eqn:fixed-pt-eqn-for-psi}
\psi + \mathcal{L}_{\vf_0}^{-1}\mathcal{N}_{\vf_0} (\psi) = - \mathcal{L}_{\vf_0}^{-1} \mathcal{F}(0).
\eeq
Note that the right hand side of \cref{eqn:fixed-pt-eqn-for-psi} is $C^{\infty}$. Moreover, using the general form of the obstruction function as described in \cref{lemma:O-tilde-form}, the definition of $\mathcal{N}_{\vf_0}$, and the fact that $\mathcal{L}_{\vf_0}^{-1}$ gains $6$ derivatives in Folland-Stein spaces, the operator $\mathrm{Id} + \mathcal{L}_{\vf_0}^{-1}\mathcal{N}_{\vf_0}$ is seen to be an elliptic zeroth order Heisenberg pseudodifferential operator (for $\vf_0$ sufficiently small). The regularity of $\psi$ therefore follows from the regularity of $ \mathcal{L}_{\vf_0}^{-1} \mathcal{F}(0)$. This concludes the proof of the theorem.
\end{proof}

\begin{remark}\label{rem:Kuranishi}
This argument also provides us with a \emph{Kuranishi map} (cf., e.g., \cite{Viaclovsky2016})
\beq
\Psi : H^1_{\mathcal{O}} \to H^2_{\mathcal{O}}
\eeq
defined near $0$ by $\Psi(\vf_0) = \mathcal{O}(\psi+\vf_0)$ where $\psi$ is obtained from $\vf_0$ as in the proof of \cref{thm:partial-solvability}; the zero set of the Kuranishi map is locally isomorphic with the set of solutions of $\mathcal{O}\equiv 0$ in a neighborhood of $0$, the isomorphism being $\vf_0 \mapsto \psi+\vf_0$ with $\psi$ as constructed above. We shall see in \cref{sec:general-rigidity-result}, however, that $0$ is an isolated point of $\Psi^{-1}(0)$, i.e. there are no nontrivial solutions of the obstruction flatness equation near the standard CR structure.
\end{remark}

\section{Second Order Deformations and Formal Rigidity}\label{sec:formal-rigidity}

In this section we consider the formal solvability of the obstruction flatness equation beyond the linearized level, and show that the problem is already not solvable at the second order; this is \cref{thm:formal-nonsolvability} in the introduction. Its proof will follow immediately from \cref{thm:formal-nonsolvability-inequality}, which is proved below. Before we state and prove \cref{thm:formal-nonsolvability-inequality} we need the following technical lemma. In the following $\norm  \cdot  \norm_s$ denotes the $L^2$ Folland-Stein Sobolev $s$-norm on the Folland-Stein space $H^s_{FS}$ of functions with $s$ ($Z_1$ and $Z_{\oneb}$) derivatives in $L^2$.
\begin{lemma} \label{lem:3-norm}
Let $\theta$ denote the standard contact form on the standard CR sphere. There is a constant $C>0$ such that for any $\vfd \in \mathfrak{D}_0^{\perp}$,
\beq
C\norm \vfd \norm_3^2 \leq  \int_{S^3} D\mathcal{Q}(\vfd) \cdot i\overline{\nabla_0\vfd} \; \theta\wedge d\theta  \leq C^{-1} \norm \vfd \norm_3^2.
\eeq
\end{lemma}
\begin{proof}
This follows immediately from \cref{lem:DQ-on-D0perp} and \cref{eqn:nabla0-on-varphi}, cf.\ \cref{rem:Foland-Stein-norms-via-spherical-harmonics}.
\end{proof}
We therefore have:
\begin{theorem}\label{thm:formal-nonsolvability-inequality}
There exists a constant $C>0$ such that for any smooth family $\vf(t)$ of deformations of the standard CR $3$-sphere with $\vf(0) = 0$ and $\left.\frac{d}{dt}\right|_{t=0} \vf(t) = \vfd \in \mathfrak{D}_0^{\perp} \setminus \{ 0 \}$ the corresponding family of obstruction functions $\mathcal{O}(t)$ (taken with respect to the standard contact form $\theta$ on $S^3$) satisfies 
\beq
\int_{S^3} \mathcal{O}(t) \,\theta\wedge d\theta \geq Ct^2 \norm \vfd \norm_3^2
\eeq
in a neighborhood of $t=0$. 
\end{theorem}
\begin{proof}
Let $\vf(t)$ be a family of deformation tensors with $\vf(0) = 0$ and $\left.\frac{d}{dt}\right|_{t=0} \vf(t) = \vfd \in \mathfrak{D}_0^{\perp} \setminus \{ 0 \}$. Then by \cref{lem:obstruction-integral},
\beq
\int_{S^3} \mathcal{O}(t) \, \theta\wedge d\theta = t^2\int_{S^3} (i\vfd_{\oneb}{}^1{}_{,0}\dot Q_1{}^{\oneb}) \, \theta\wedge d\theta + O(t^3),
\eeq
where $\dot Q_1{}^{\oneb} = D\mathcal{Q}(\vfd_{1}{}^{\oneb})$ and $\vfd_{\oneb}{}^1 = \overline{\vfd_{1}{}^{\oneb}} = \overline{\vfd}$. The result then follows immediately from \cref{lem:3-norm}.
\end{proof}
\begin{remark} Note that \cref{thm:formal-nonsolvability-inequality} is not claiming that the integral of the obstruction function becomes positive whenever we slightly deform the CR structure of the standard CR sphere, but only that it does so when the deformation starts out in an infinitesimally obstruction flat direction. 
\end{remark}

\section{General Rigidity Result}\label{sec:general-rigidity-result}
In this section we prove our main rigidity result, which implies \cref{thm:sphere-is-rigid}.
\begin{theorem}\label{thm:rigidity}
There is an open neighborhood $U$ of the origin in the slice $\mathfrak{D}'_{BE}\oplus \mathfrak{D}_0^{\perp}$ (in the $H^3_{FS}$ topology) such that the CR structure corresponding to a deformation tensor $\vf \in U$ is obstruction flat if and only if $\vf =0$.
\end{theorem}
\begin{proof}
With a view to obtaining a contradiction we suppose that there exists a sequence $\vf^{(k)}$ of nonzero obstruction flat deformation tensors such that $\vf^{(k)}\to 0$ in $H^3_{FS}$. Set $\vf^{(k)} = \epsilon_k \hat{\vf}^{(k)}$, where $\norm\hat{\vf}^{(k)}\norm_3=1$. Then $\epsilon_k \to 0$. We let $\mathcal{P}_1: \mathfrak{D}'_{BE}\oplus \mathfrak{D}_0^{\perp} \to \mathfrak{D}'_{BE}$ and $\mathcal{P}_2: \mathfrak{D}'_{BE}\oplus \mathfrak{D}_0^{\perp} \to \mathfrak{D}_0^{\perp}$ denote the $L^2$ orthogonal projections, and write 
\beq
\hat{\vf}^{(k)} = \hat{\vf}^{(k)}_{\mathfrak{D}'_{BE}} + \hat{\vf}^{(k)}_{\mathfrak{D}_0^{\perp}}
\eeq
where $\hat{\vf}^{(k)}_{\mathfrak{D}'_{BE}} = \mathcal{P}_1 \hat{\vf}^{(k)}$ and $\hat{\vf}^{(k)}_{\mathfrak{D}_0^{\perp}} = \mathcal{P}_2 \hat{\vf}^{(k)}$. Our first goal is to show that $\bnorm \hat{\vf}^{(k)}_{\mathfrak{D}'_{BE}}  \bnorm_3 \to 0$ (and hence $\bnorm \hat{\vf}^{(k)}_{\mathfrak{D}_0^{\perp}}  \bnorm_3 \to 1$) as $k\to \infty$, meaning that $\vf^{(k)}$ is approximately a solution of the linearized equation for large $k$.

As above we fix $\theta$, the standard contact form on the $3$-sphere, and for each $\vf^{(k)}$ we compute with respect to the frame given by \cref{eqn:tilde-Z1} and the corresponding admissible coframe. We let $Q_{11}^{(k)}$ denote the Cartan umbilical tensor of the CR structure with deformation tensor $\vf^{(k)}$, $\nabla_{(k)}$ the Tanaka-Webster connection of the corresponding pseudohermitian structure, and $A^{11}_{(k)}$ its pseudohermitian torsion. Since each $\vf^{(k)}$ is assumed to be obstruction flat, we have
\beq \label{eqn:obsruction-flat-k}
\nabla^1_{(k)}\!\nabla^1_{(k)} Q_{11}^{(k)} = iA^{11}_{(k)} Q_{11}^{(k)}
\eeq
for each $k$. Note that the part of $\nabla^1_{(k)}\!\nabla^1_{(k)} Q_{11}^{(k)}$ that is linear in $\vf^{(k)}$ is $D\mathcal{O}(\vf^{(k)}) = (Z_{\oneb})^2 D\mathcal{Q}( \vf^{(k)})$, which equals $\epsilon_k (Z_{\oneb})^2 D\mathcal{Q}\left( \hat{\vf}^{(k)}_{\mathfrak{D}'_{BE}}\right)$ since $\hat{\vf}^{(k)}_{\mathfrak{D}_0^{\perp}}$ solves the linearized equation. Moving all nonlinear terms in \cref{eqn:obsruction-flat-k} to the right hand side we may rewrite the equation as
\beq \label{eqn:obsruction-flat-k-linear-vs-nonlinear}
\epsilon_k (Z_{\oneb})^2 D\mathcal{Q}\left( \hat{\vf}^{(k)}_{\mathfrak{D}'_{BE}}\right) = \mathcal{F}(\vf^{(k)}),
\eeq
where $\mathcal{F}(\vf)$ is given by $(1-|\vf|^2)^{-9}$ times a polynomial in $\vf$, $\overline{\vf}$ and their $Z_1$ and $Z_{\oneb}$ derivatives up to order $6$ having no linear term, cf.\  \cref{lemma:O-tilde-form}. One might expect that $\mathcal{F}(\vf^{(k)})$ is $O(\epsilon_k^2)$, from which it would follow that $\bnorm\hat{\vf}^{(k)}_{\mathfrak{D}'_{BE}}\bnorm_6 \to 0$. The problem with this is that we only assumed that $\norm \vf^{(k)} \norm_3 \to 0$ while $\mathcal{F}(\vf^{(k)})$ involves derivatives up to order $6$. Since the idea is to show that the integral of the obstruction function must be nonzero for large $k$, it is natural to assume $\norm \vf^{(k)} \norm_3 \to 0$ rather than $\norm \vf^{(k)} \norm_6 \to 0$; cf.\ \cref{lem:obstruction-integral} and \cref{lem:3-norm} which will be used below. 

By \cref{lem:P1DQ-on-D'BE,lem:P1DQ-vs-sublaplacian-squared,lem:Z1-squared-Zoneb-squared-in-spherical-harmonics} (cf.\ \cref{rem:Foland-Stein-norms-via-spherical-harmonics}) there exists a constant $C>0$ such that for $\vf \in \mathfrak{D}'_{BE}$,
\beq
C\bnorm \vf \bnorm_3^2 \leq  \int_{S^3} \overline{\vf}\cdot (1+\Delta_b)^{-1}(Z_1)^2(Z_{\oneb})^2 D\mathcal{Q}\left( \vf \right)\; \theta \wedge d\theta \leq C^{-1} \bnorm \vf \bnorm_3^2.
\eeq
Applying $(1+\Delta_b)^{-1}(Z_1)^2$ to both sides of \cref{eqn:obsruction-flat-k-linear-vs-nonlinear} and then integrating the resulting expression against the conjugate of $\hat{\vf}^{(k)}_{\mathfrak{D}'_{BE}}$ one therefore obtains that
\beq\label{eqn:bound-on-norm-of-phi-D'BE-from-equation}
\epsilon_k C\bnorm \hat{\vf}^{(k)}_{\mathfrak{D}'_{BE}} \bnorm_3^2 \leq \int_{S^3} \overline{\hat{\vf}^{(k)}_{\mathfrak{D}'_{BE}}} \cdot (1+\Delta_b)^{-1}(Z_1)^2\mathcal{F}(\vf^{(k)})\; \theta\wedge d \theta .
\eeq
Let $W_{FS}^{s,p}$ denote the Folland-Stein space of functions with $s$ derivatives in $L^p$ (where the derivatives are $Z_1$ and $Z_{\oneb}$ derivatives). By using integration by parts to balance the numbers of derivatives on the factors, applying the generalized H\"older inequality, and then using the Sobolev embedding theorem for Folland-Stein spaces \cite{FollandStein1974,Folland1975} (in particular that, since the homogeneous dimension of the CR $3$-sphere is 4, the $H^3_{FS}$ norm controls the $W_{FS}^{2,4}$ norm, the $W_{FS}^{1,q}$ norm for $2\leq q < \infty$, and the $L^{\infty}$ norm) using \cref{lemma:Z1Z1O-tilde-form} one can show that
\beq\label{eqn:nonlinear-terms-controlled}
\int_{S^3} \overline{\hat{\vf}^{(k)}_{\mathfrak{D}'_{BE}}} \cdot (1+\Delta_b)^{-1}(Z_1)^2\mathcal{F}(\vf^{(k)})\; \theta\wedge d \theta  = O(\epsilon_k^2)
\eeq
as $k\to \infty$ (since $\mathcal{F}(\vf)$ has no linear terms).
We illustrate this by showing how to estimate the term in the left hand side of \cref{eqn:nonlinear-terms-controlled} that arises from the term $\frac{1}{6}(1-|\vf|^2)^{-9}\vf_{\oneb}{}^1\vf_1{}^{\oneb}{}_{,}{}^{1\oneb}{}_{11\oneb}{}^{1}$ in $\mathcal{F}(\vf)$: To show that
\beq\label{eqn:example-nonlinear-terms-O-epsilon-squared}
 \int_{S^3} \overline{\hat{\vf}^{(k)}_{\mathfrak{D}'_{BE}}} \cdot (1+\Delta_b)^{-1}(Z_1)^2\left((1-|\vf^{(k)}|^2)^{-9}\vf^{(k)}{}_{\oneb}{}^1\vf^{(k)}{}_1{}^{\oneb}{}_{,}{}^{1\oneb}{}_{11\oneb}{}^{1} \right)\; \theta\wedge d \theta  = O(\epsilon_k^2)
\eeq
we first write $\vf^{(k)}{}_{\oneb}{}^1\,\vf^{(k)}{}_1{}^{\oneb}{}_{,}{}^{1\oneb}{}_{11\oneb}{}^{1}=\epsilon_k^2 \,\hat\vf^{(k)}{}_{\oneb}{}^1\,\hat\vf^{(k)}{}_1{}^{\oneb}{}_{,}{}^{1\oneb}{}_{11\oneb}{}^{1}$ and integrate $(1+\Delta_b)^{-1}(Z_1)^2$ by parts so that the left hand side of \cref{eqn:example-nonlinear-terms-O-epsilon-squared} becomes $\epsilon_k^2$ times
\beq\label{eqn:example-nonlinear-terms}
 \int_{S^3} \left((Z_1)^2(1+\Delta_b)^{-1}\overline{\hat{\vf}^{(k)}_{\mathfrak{D}'_{BE}}}\,\right) \cdot (1-|\vf^{(k)}|^2)^{-9}\hat\vf^{(k)}{}_{\oneb}{}^1\hat\vf^{(k)}{}_1{}^{\oneb}{}_{,}{}^{1\oneb}{}_{11\oneb}{}^{1} \; \theta\wedge d \theta.
\eeq
We then show that \cref{eqn:example-nonlinear-terms} is bounded. To see this we first integrate by parts to remove three of the derivatives from the $\hat\vf_1{}^{\oneb}{}_{,}{}^{1\oneb}{}_{11\oneb}{}^{1}$ factor. In this way, by the Leibniz rule, we obtain the integral of $(1-|\vf|^2)^{-9}$ times a polynomial in $\hat\vf^{(k)}$, $\overline{\hat\vf^{(k)}}$, $(Z_1)^2(1+\Delta_b)^{-1}\overline{\hat{\vf}^{(k)}_{\mathfrak{D}'_{BE}}}$ and their $Z_1$ and $Z_{\oneb}$ derivatives up to order at most $3$ (note that the Folland-Stein $3$-norm of $(Z_1)^2(1+\Delta_b)^{-1}\overline{\hat{\vf}^{(k)}_{\mathfrak{D}'_{BE}}}$ is uniformly bounded since $(1+\Delta_b)^{-1}$ gains two derivatives in Folland-Stein spaces), where the total number of $Z_1$ and $Z_{\oneb}$ derivatives on the factors in any term of the polynomial is $6$. Since the Folland-Stein $3$-norm controls the $L^{\infty}$ norm we need only concern ourselves with the factors in each term of the polynomial which have at least one $Z_1$ or $Z_{\oneb}$ derivative. Since the total number of derivatives in each term is $6$ and one term has $3$ derivatives, the (nonzero) numbers of derivatives must be $3+3$, $3+2+1$ or $3+1+1+1$; in each case, by the generalized H\"older inequality (with the partitions $1=\frac{1}{2}+\frac{1}{2}$, $1 = \frac{1}{2}+\frac{1}{4}+\frac{1}{4}$, $1 = \frac{1}{2}+\frac{1}{6}+\frac{1}{6}+\frac{1}{6}$ for the three respective cases) we can estimate the term arising in the integral in terms of the $H_{FS}^3$ norm, the $W_{FS}^{2,4}$ norm, the $W_{FS}^{1,6}$ norm, and the $L^{\infty}$ norm of $\vf^{(k)}$. As commented above, these norms are all uniformly bounded since the $H_{FS}^3$ norm is uniformly bounded. This shows that \cref{eqn:example-nonlinear-terms} is uniformly bounded in $k$, giving \cref{eqn:example-nonlinear-terms-O-epsilon-squared}. Arguing similarly for the other terms in $\mathcal{F}(\vf)$, of which there are finitely many, one obtains \cref{eqn:nonlinear-terms-controlled}. Hence, by \cref{eqn:bound-on-norm-of-phi-D'BE-from-equation} we have
\beq\label{eqn:key-estimate}
\bnorm \hat{\vf}^{(k)}_{\mathfrak{D}'_{BE}} \bnorm_3 = O(\epsilon_k^{1/2}).
\eeq
It follows that $\bnorm \hat{\vf}^{(k)}_{\mathfrak{D}_0^{\perp}}  \bnorm_3 \to 1$. 


Now since we have assumed $\mathcal{O}^{(k)} = \nabla^1_{(k)}\!\nabla^1_{(k)} Q_{11}^{(k)} - iA^{11}_{(k)} Q_{11}^{(k)} =0$ for each $k$, by \cref{lem:obstruction-integral} followed by \cref{lemma:Q11-tilde-form} and \cref{eqn:key-estimate} we have
\begin{align}
0= \int_{S^3} \mathcal{O}^{(k)} \, \theta\wedge d\theta &= i\int_{S^3} \dfrac{\vf^{(k)}{}_{\oneb}{}^1{}_{,0}Q^{(k)}{}_1{}^{\oneb} }{1-|\vf^{(k)}|^2} \, \theta\wedge d\theta \\
\nonumber &= \epsilon_k^2\int_{S^3} D\mathcal{Q}(\hat\vf^{(k)}_{\mathfrak{D}_0^{\perp}}) \cdot i\overline{\nabla_0\hat\vf^{(k)}_{\mathfrak{D}_0^{\perp}}} \;\theta\wedge d\theta + O(\epsilon_k^{5/2}).
\end{align}
It follows that 
\beq \label{eqn:contradiction-a}
\lim_{k\to \infty} \int_{S^3} D\mathcal{Q}(\hat\vf^{(k)}_{\mathfrak{D}_0^{\perp}}) \cdot i\overline{\nabla_0\hat\vf^{(k)}_{\mathfrak{D}_0^{\perp}}} \;\theta\wedge d\theta = \lim_{k\to\infty} \frac{1}{\epsilon_k^2} \int_{S^3} \mathcal{O}^{(k)} \, \theta\wedge d\theta = 0.
\eeq 
On the other hand \cref{lem:3-norm} and the fact that $\bnorm \hat{\vf}^{(k)}_{\mathfrak{D}_0^{\perp}}  \bnorm_3 \to 1$ imply that 
\beq
\lim_{k\to \infty} \int_{S^3} D\mathcal{Q}(\hat\vf^{(k)}_{\mathfrak{D}_0^{\perp}}) \cdot i\overline{\nabla_0\hat\vf^{(k)}_{\mathfrak{D}_0^{\perp}}} \;\theta\wedge d\theta \geq \lim_{k\to \infty}C\bnorm \hat{\vf}^{(k)}_{\mathfrak{D}_0^{\perp}} \bnorm_3 = C > 0,
\eeq
contradicting \cref{eqn:contradiction-a}. We conclude that there is no such sequence $\vf^{(k)}$ of nonzero obstruction flat deformation tensors.
\end{proof}

\bibliographystyle{plain}

\newcommand{\noopsort}[1]{}

\end{document}